%% file: AV_Glaubitz_arxiv.tex
\documentclass[a4paper, 10pt]{scrartcl}
\pdfoutput=1

\include{includes_arXiv}

\title{Smooth and compactly supported viscous sub-cell shock capturing for Discontinuous Galerkin 
methods}
\author{ 
  J. Glaubitz\thanks{%
    \textit{Corresponding author:} Jan Glaubitz \newline 
    Max-Planck-Institut f\"ur Mathematik, Vivatsgasse 7, 53111 Bonn, Germany. \newline
    Institut Computational Mathematics, TU Braunschweig, Universit\"atsplatz 2, 38106 Braunschweig, 
Germany. \newline
    ( \href{mailto:j.glaubitz@tu-bs.de}{j.glaubitz@tu-bs.de} )
  }, \
  A.C. Nogueira Jr., 
  J.L.S. Almeida, 
  R.F. Cant\~ao, and
  C.A.C. Silva
}
\date{\monthyeardate\today}

\begin{document}

\maketitle

\begin{abstract}
  \input{0_Abstract}
\end{abstract}



\input{1_introduction}
\input{2_discretisation}
\input{3_design}
\input{4_motivation}
\input{5_modal}
\input{6_sensor}
\input{7_tests}
\input{Summary}

\acknowledgement{This work developed during a two-moth stay of the first author at the Max Planck 
Institute for Mathematics (MPIM) in Bonn during summer of 2017. 
He would like to express his gratitude for the generous financial support by the MPIM as well as 
the warm and inspiring research atmosphere provided by its staff. 
Further, we would like to thank the anonymous referees for many helpful suggestions.
}  

\bibliographystyle{abbrv}
\bibliography{literature}

\end{document}

%% file: includes_arXiv.tex
\usepackage[utf8]{luainputenc}
\usepackage[english]{babel}
\usepackage{csquotes}

\usepackage[a4paper, textheight=700pt, textwidth=468pt, centering]{geometry}

\usepackage[plainpages=false,pdfpagelabels,hidelinks]{hyperref}
\makeatletter
\hypersetup{pdfauthor={\@author}}
\hypersetup{pdftitle={\@title}}
\makeatother

\usepackage{datetime}
\newdateformat{monthyeardate}{%
  \monthname[\THEMONTH], \THEYEAR}

\usepackage{cite}

\usepackage{amsmath}
\allowdisplaybreaks
\usepackage{amssymb}
\usepackage{commath}
\usepackage{mathtools}
\usepackage{siunitx} 
\usepackage{algorithm}
\usepackage[noend]{algpseudocode}

\usepackage{amsthm}
\theoremstyle{plain}
  \newtheorem{theorem}{Theorem}[section]
  
  \theoremstyle{definition}
  
\newtheorem{remark}[theorem]{Remark}

\usepackage{color}
\usepackage{graphicx}
\usepackage[small]{caption}
\usepackage{subcaption}

\usepackage{pgfplots}
\pgfplotsset{compat=1.11}
\usepgfplotslibrary{external} 
\begingroup\expandafter\expandafter\expandafter\endgroup
\expandafter\ifx\csname pdfsuppresswarningpagegroup\endcsname\relax
\else
\fi

\usepackage{comment}

\usepackage{booktabs}
\usepackage{rotating}

\usepackage{calc}
\usepackage{xparse}

\newcommand{\subjclass}[1]{\textbf{MSC2010:} \ #1}
\newcommand{\keywords}[1]{\textbf{Keywords:} \ #1} 
\newcommand{\acknowledgement}[1]{\bigskip \textbf{Acknowledgement.} \ #1}
\renewcommand{\vec}[1]{\underline{#1}}
\NewDocumentCommand{\mat}{mo}{%
  \IfValueTF{#2}{%
    \underline{\underline{#1}}{#2}
  }{%
    \underline{\underline{#1}}\,
  }%
}

\newcommand{\scp}[2]{\left\langle{#1,\, #2}\right\rangle}
\renewcommand{\d}{\operatorname{d}}

\newcommand{\fnum}{f^{\mathrm{num}}}
\newcommand{\vecfnum}{\vec{f}^{\mathrm{num}}}

\renewcommand{\epsilon}{\varepsilon}
\renewcommand{\phi}{\varphi}

\newcommand{\R}{\mathbb{R}}

\newsavebox{\DelimiterBox}
\newlength{\DelimiterHeight}
\newlength{\DelimiterDepth}
\newsavebox{\ArgumentBox}
\newlength{\ArgumentHeight}
\newlength{\ArgumentDepth}
\newlength{\ResizedDelimiterHeight}

%% file: 0_Abstract.tex
In this work, a novel artificial viscosity method is proposed using smooth and compactly supported viscosities. 
These are derived by revisiting the widely used piecewise constant artificial viscosity method of Persson and Peraire as well as the piecewise linear refinement of Kl\"ockner et al. with respect to the fundamental design criteria of conservation and entropy stability. 
Further investigating the method of modal filtering in the process, it is demonstrated that this strategy has inherent shortcomings, which are related to problems of Legendre viscosities to handle shocks near element boundaries. 
This problem is overcome by introducing certain functions from the fields of robust reprojection and 
mollififers as viscosity distributions. 
To the best of our knowledge, this is proposed for the first time in this work. 
The resulting \emph{$C_0^\infty$ artificial viscosity method} is demonstrated to provide sharper profiles, steeper  gradients and a higher resolution of small-scale features while still maintaining stability of the method. 

%% file: 1_introduction.tex
\section{Introduction}
\label{sec:introduction}

In the last decades, great efforts have been made to develop accurate and stable numerical methods for time dependent partial differential equations (PDEs), especially for hyperbolic conservation laws. 
Traditionally, low order numerical schemes, have been used 
to solve hyperbolic conservation laws, in particular in industrial applications.  
But since they become quite costly for high accuracy or long time simulations, there is a rising 
demand for high order methods.
Such high order methods, like flux reconstruction \cite{huynh2007flux,vincent2011new} or several 
Discontinuous Galerkin schemes \cite{hesthaven2007nodal,gassner2013skew,glaubitz2017novel}, 
typically use polynomial approximations to the solution. 
At least for smooth solutions, they are capable of reaching spectral orders of accuracy.
Yet, special care has to be taken to the fact that solutions of hyperbolic conservation laws might form spontaneous discontinuities. 
Due to the Gibbs phenomenon, polynomial approximations of jump functions typically show spurious oscillations and yield the underlying numerical scheme to break down.
While low-order methods add too much dissipation to the numerical solution, hence smearing 
smaller scaled features, high-order methods add too little. 
Therefore, in recent hp-methods, the idea is to add artificial dissipation in elements where discontinuities arise. 
Such procedures are known as \emph{shock capturing techniques}.

In this work, a novel artificial viscosity method is proposed for sub-cell shock capturing in 
Discontinuous Galerkin and related methods. 
This new artificial viscosity method, referred to as the \emph{$C_0^\infty$ artificial viscosity 
method}, essentially relies on the idea to replace commonly used viscosities in the artificial 
viscosity method \cite{persson2006sub,klockner2011viscous} by certain weight functions from the 
field of robust reprojection \cite{gelb2006robust} and mollifiers \cite{tadmor2002adaptive}. 

Further, the novel $C_0^\infty$ artificial viscosity method is derived by revisiting the most commonly used existing viscosity methods, such as the piecewise constant artificial viscosity method of Persson and Peraire \cite{persson2006sub} as well as the piecewise linear method of Kl\"ockner et al. \cite{klockner2011viscous}, with regard to the fundamental design criteria of conservation and entropy stability. 
To the best of the authors' knowledge, this is the first work to investigate these methods in a 
strictly analytical sense.
In the process, we are able to pinpoint certain drawbacks of these methods, formulate precise 
criteria on the viscosity terms for certain properties to hold, and thus construct novel ones with 
favorable properties. 

It should be noted that the related strategy of modal filtering is addressed, which is a 
widely used tool 
\cite{gottlieb2001spectral,boyd2001chebyshev,hesthaven2008filtering,glaubitz2016application}, since 
it promises to be somewhat of a more efficient and easy to implement 
formulation of the artificial viscosity method. 
We will show, however, that modal filtering has inherent drawbacks which can not be overcome. 
This will be demonstrated by showing that modal filtering (at least for exponential filters) 
corresponds to specific Legendre or more general Jacobi artificial viscosities which again are 
highlighted to perform poor when shock discontinuities form near element boundaries.
To the best of our knowledge, this is the first time the whole class of modal filtering (by exponential filters) is exposed to hold such a shortcoming. 

The rest of this work is organised as follows. 
In section \ref{sec:discretisation} a short description of the Discontinuous Galerkin method is 
given.
Section \ref{sec:design} exposes the essential design criteria of conservation and entropy stability for hyperbolic conservation laws on which the subsequent theoretical investigation of different artificial viscosity methods are based. 
The artificial viscosity method is then introduced in section \ref{sec:motivation} and the most widely used versions of Persson and Peraire as well as of Kl\"ockner et al. are revisited.  
In this section it is further proved that conservation holds for the viscosity extension of a conservation law, once the artificial viscosity is continuous and compactly supported. 
In section \ref{sec:modal} this discussion is extended to entropy stability of viscosity extensions, 
where it is shown that the artificial viscosity in addition has to be nonnegative for entropy 
stability to carry over. 
Moreover pinpointing the crucial drawbacks of modal filtering, we utilize the acquired design criteria for artificial viscosity terms to propose novel ones.
Section \ref{sec:sensor} introduces the shock sensor which will be used for the subsequent numerical tests in order to steer the location and strength of dissipation added by the novel artificial viscosity method. 
Numerical tests demonstrating the performance of the proposed novel strategies - also in comparison with the commonly used one of Kl\"ockner et al. - are provided in section \ref{sec:tests} for the system of Euler equations. 
The work closes in section \ref{sec:summary} by summarizing the characteristic features
of the proposed new artificial viscosity methods and discussing possible future applications.

%% file: 2_discretisation.tex
\section{Discretization}
\label{sec:discretisation}

Traditionally, the DG approach combines ideas from Finite Volume (FV) and Spectral Element Methods (SEM). 
When space and time are decoupled by the method of lines \cite{leveque2002finite}, the DG method is designed as a semidiscretization of hyperbolic systems of conservation laws, 
\begin{equation}
    \partial_t u + \nabla \cdot F(u) = 0,
\end{equation}
with appropriate initial and boundary conditions on a physical domain $\Omega$, which is decomposed into $I$ disjoint, face-conforming elements $\Omega_i$, $\Omega = \bigcup_{i=1}^I \Omega_i$.

The DG method in one space dimension uses a nodal or modal polynomial basis of order $p$ in a reference element $\Omega_{ref} = [-1,1]$. 
The $I$ elements are mapped to this reference element, where all computations are performed then. 
The extension to multiple dimensions can be achieved by tensor products. 
Thus, in the following, just this one dimensional case is briefly revised. 
For a more detailed as well as more general description of DG schemes, see for instance the works 
\cite{hesthaven2007nodal,karniadakis2013spectral} and references cited therein. 

In one space dimension, the DG approach uses polynomial approximations $u_p,f_p \in \mathbb{P}_p([-1,1])$ to the solution $u(t,x)$ respectively the flux $f(u)$ at time $t$. 
It is constructed by the approach of the residual 
\begin{equation}
    \mathcal{R}_p(t,x) := \partial_t u_p + \partial_x f_p 
\end{equation}
to vanish in the sense that it is orthogonal to all local test functions $\psi \in \mathbb{P}_p([-1,1])$, i.e. 
\begin{equation}
    \int_{-1}^1 \mathcal{R}_p(t,x) \psi(x) \d x = 0 
    \quad \forall \psi \in \mathbb{P}_p([-1,1]).
\end{equation}
By applying integration by parts once, the locally defined weak form 
\begin{equation} 
    \int_{-1}^1 \left( \partial_t u_p \right) \psi(x) \d x - \int_{-1}^1 f_p \left( \partial_x \psi(x) \right) \d x
    = - \left( \fnum_R \psi(1) - \fnum_L \psi(-1) \right)
\end{equation}
is recovered, and by applying integration by parts a second time, the locally defined strong form 
\begin{equation} \small
    \int_{-1}^1 \left( \partial_t u_p + \partial_x f_p \right) \psi(x) \d x
    = \left( f_p(1) - \fnum_R \right) \psi(1) - \left( f_p(-1) - \fnum_L \right) \psi(-1)
\end{equation}
arises. 
Here, $\fnum$ is a suitably chosen numerical flux.
Representing the numerical solution $u_p$, the flux (reconstruction) $f_p$, and the test functions $\psi$ all with respect to the same basis $\{ \phi_k \}$,  
\begin{equation}
    u_p(t,x) = \sum_{k=0}^p \hat{u}_k(t) \phi_k(x), \quad 
    f_p(x) = \sum_{k=0}^p \hat{f}_k \phi_k(x), \quad
    \psi(x) = \sum_{j=0}^p \hat{\psi}_j \phi_j(x),
\end{equation}
$p+1$ equations
\begin{equation}\label{eq:DG} 
  \begin{aligned}
    \sum_{k=0}^p & \int_{-1}^1 \left( \partial_t \hat{u}_k(t) \phi_k(x) + \hat{f}_k \partial_x 
\phi_k(x) \right) \phi_j(x) \d x \\ 
    & = \sum_{k=0}^p \left( \hat{f}_k \phi_n(1) - \fnum_R \right) \phi_j(1) - 
    \left( \hat{f}_k \phi_k(-1) - \fnum_L \right) \phi_j(-1),
  \end{aligned}        
\end{equation}
for $j=0,\dots,p$, are obtained for the $p+1$ unknown coefficients $\hat{u}_k$ of the numerical solution on each element. 

By utilizing the mass, stiffness, differentiation, restriction (interpolation), and boundary
integral matrices
\begin{align}
    \mat{M}_{jk} & = \int_{-1}^1 \psi_j \psi_k \d x, \quad 
    \mat{S}_{jk} = \int_{-1}^1 \psi_j \partial_{x} \psi_k \d x, \quad
    \mat{D} = \mat{M}^{-1} \mat{S}, \\
    \mat{R} & = 
    \begin{pmatrix} 
        \phi_0(1) & \dots & \phi_p(1) \\ \phi_0(-1) & \dots & \phi_p(-1) 
    \end{pmatrix}, \qquad 
    \mat{B} = \begin{pmatrix} 
        -1 & 0 \\ 0 & 1 
    \end{pmatrix},
\end{align}
the DG projection \eqref{eq:DG} can be rewritten in its matrix form 
\begin{equation}
    \partial_t \vec{u} = - \mat{D} \vec{f} + \mat{M}^{-1} \mat{R}^T \mat{B} \left( \mat{R} \vec{f} - \vecfnum \right).
\end{equation}
Here, $\vec{u},\vec{f}$ are the vectors containing the coefficients of $u_p, f_p$ and $\vecfnum$ is the vector containing the values of the numerical flux at the element boundaries.

The resulting system of ordinary differential equations for the solution coefficients $\vec{u}$, is 
integrated in time using for instance a Runge-Kutta method then. 
See the extensive literature \cite{gottlieb1998total,levy1998semidiscrete,gottlieb2001strong,ketcheson2008highly,gottlieb2011strong}. 
In this work, the later numerical tests are obtained by the explicit strong stability preserving (SSP) Runge-Kutta (RK) method of third order using three stages (SSPRK(3,3)) or of forth order using five stages (SSPRK(4,5)) given by Gottlieb and Shu in \cite{gottlieb1998total} and Hesthaven and Warburton in \cite{hesthaven2007nodal}, respectively.

%% file: 3_design.tex
\section{Design principles and equations}
\label{sec:design}

When space and time are decoupled by the method of lines, numerical methods for hyperbolic conservation laws $\partial_t u = - \partial_x f(u)$ are essentially designed as semidiscretisations of the right hand side $- \partial_x f(u)$. 
In the DG approach described in section \ref{sec:discretisation}, this semidiscretisation is given 
by 
\begin{equation}
  - \mat{D} \vec{f} + \mat{M}^{-1} \mat{R}^T \mat{B} \left( \mat{R} \vec{f} - \vecfnum \right).
\end{equation} 
But how does one construct such a semidiscretisation? 
Besides consistency with the underlying differential equation, numerical methods typically aim to mimic certain properties of the (unknown) analytical solution. 
In this work, the focus will be given to the very essential properties of conservation and entropy stability as design principles for numerical schemes. 
In fact, a broad class of artificial viscosity terms for sub-cell shock capturing methods 
will be 
rejected, since they already violated conservation. 
Furthermore, every hyperbolic conservation law provides its very own entropy (criterion), 
which is then used to select solutions of physical significance.  
In this section, the most important equations and associated entropy functions are briefly revisited.

\subsection{Design principles} 

Conservation laws 
\begin{equation}\label{eq:scalar-CL}
    \partial_t u + \partial_x f(u) = 0
\end{equation}
ensure that the rate of change of the total amount of a particular measurable property $u$ in a fixed domain $\Omega$ is equal to the flux of that property across the boundary of $\Omega$, i.e. 
\begin{equation}\label{eq:conservation}
  \od{}{t} \int_\Omega u
  = - f(u) \big|_{\partial \Omega}.
\end{equation}

Recalling \cite{lax1973hyperbolic} that a solution of \eqref{eq:scalar-CL} may develop spontaneous jump discontinuities (shock waves) in finite time and even for smooth initial condition, the more general class of weak solutions  has to be permitted. 
However, since there are many possible weak solutions then, the one of physical significance must be selected. 
This is done by augmenting the conservation law \eqref{eq:scalar-CL} with an additional entropy condition which requires 
\begin{equation}\label{eg:entropy-eq}
    \partial_t U(u) + \partial_x F(u) \leq 0.
\end{equation}
$U(u)$ and $F(u)$ are any strictly convex entropy function and corresponding entropy-flux associated with the conservation law \eqref{eq:scalar-CL} in the sense that they have to satisfy $U'(u) f'(u) = F'(u)$. 
A strict inequality in \eqref{eg:entropy-eq} reflects the existence of physically relevant shock 
waves in the solution of the system \eqref{eq:scalar-CL}, \eqref{eg:entropy-eq} then. 
Similar to the associated conservation law, the entropy condition \eqref{eg:entropy-eq} ensures that the rate of change of the total amount of entropy $U$ in a fixed domain $\Omega$ is equal or less to the entropy-flux across the boundary of $\Omega$, i.e. 
\begin{equation}\label{eq:stability}
  \od{}{t} \int_\Omega U(u)
  \leq - F(u) \big|_{\partial \Omega}.
\end{equation}
Especially for periodic boundary conditions, the two mayor design principles,
\begin{align}
    \text{conservation:} \qquad & \od{}{t} \int_\Omega u = 0, \label{eq:design_con} \\ 
    \text{entropy stability:} \qquad & \od{}{t} \int_\Omega U(u) \leq 0, \label{eq:design_ent}
\end{align}
which will be further utilised in section \ref{sec:motivation} and \ref{sec:modal}, arise.

\subsection{Linear advection and Euler equations}

At the very simple end of the spectrum of hyperbolic conservation laws lies the \emph{linear 
advection equation}, 
\begin{equation}\label{eq:linear_advection}
    \partial_t u + \partial_x u = 0,
\end{equation}
with constant velocity. 
Equation \eqref{eq:linear_advection} transports its initial condition in time with speed $1$. 
Thus, when the equation is augmented with an initial condition $u_0(x)$, the exact solution is simply given by $u(t,x) = u_0(x-t)$. 
For the linear advection equation, one has $f'(u) = 1$ and hence $U'(u) = F'(u)$ for a pair $(U,F)$ of a strictly convex entropy $U$ and an entropy-flux $F$. 
One such pair associated to the linear advection equation, is therefore already given by 
\begin{equation}\label{eq:entropy_advection}
    U(u) = F(u) = \frac{1}{2} u^2,
\end{equation} 
which is the so called \emph{$L^2$ entropy}. 
In fact, for scalar conservation laws, any strictly convex function $U$ is an entropy 
\cite{godlewski1991hyperbolic}.
Note that it is much more difficult to find an entropy $U$ in the general case of systems. 
In fact, the existence of entropy functions is a special property of the system. 
However, in all practical examples derived from Mechanics or Physics, finding an entropy which has a physical meaning is possible. 

Such a practical example are the \emph{Euler equations} 
\begin{equation}\label{eq:euler}
    \partial_t 
    \underbrace{\begin{pmatrix} \rho \\ m \\ E \end{pmatrix}}_{=u} 
    + \partial_x 
    \underbrace{\begin{pmatrix} m \\ vm + \rho P \\ v (E + P) \end{pmatrix}}_{=f(u)}
    = 0,
\end{equation}
of gas dynamics. 
Here, $\rho$ is the density, $v$ is the velocity, $E$ is the total energy, $m = q \rho$ is the momentum, $P = (\gamma - 1)(E-\frac{1}{2}v^2 \rho)$ is the pressure, and $\gamma$ is the ratio of specific heats. 
Euler equations broadly apply to compressible, inviscid flow problems.
An entropy-entropy flux pair $(U,F)$ associated to the Euler equations \eqref{eq:euler} is given by 
\begin{equation}\label{eq:entropy_euler}
    U = - \rho s, \qquad 
    F = - \rho s u,
\end{equation}
where $s$ is the physical entropy
\begin{equation*}
    s = \ln(P \rho^{-\gamma}) + C = \ln(P) - \ln(\rho^\gamma) + C
\end{equation*}
with constant term $C \in \R$, which can be ignored.
In fact, there are many possible choices for the entropy function, such as $U = \rho \mu(s)$ for any convex function $\mu$, see \cite{harten1983symmetric}. 
The above choice, however, is the only one which is consistent with the entropy condition from thermodynamics \cite{hughes1986new} in the presence of heat transfer. 

%% file: 4_motivation.tex
\section{\texorpdfstring{Motivation for $C^0$ artificial viscosity}{Motivation for C0 artificial 
viscosity}}
\label{sec:motivation}

Spontaneous development of jump discontinuities in solutions of hyperbolic conservation laws is not only a challenge from a theoretical point of view but also from a numerical. 
In addition to selecting physical significant solutions, spurious oscillations arise in the numerical solution when shock waves are present. 
Especially high order schemes, where the 'wiggles' stem from the Gibbs phenomenon 
\cite{hewitt1979gibbs}, often become unstable and might finally break down. 
This is illustrated in Figure \ref{fig:linear_SSPRK33} for the linear advection equation 
\eqref{eq:linear_advection} with a square wave as initial condition.

\begin{figure}[!htb]
  \centering
  \begin{subfigure}[b]{0.49\textwidth}
    \includegraphics[width=\textwidth]{%
      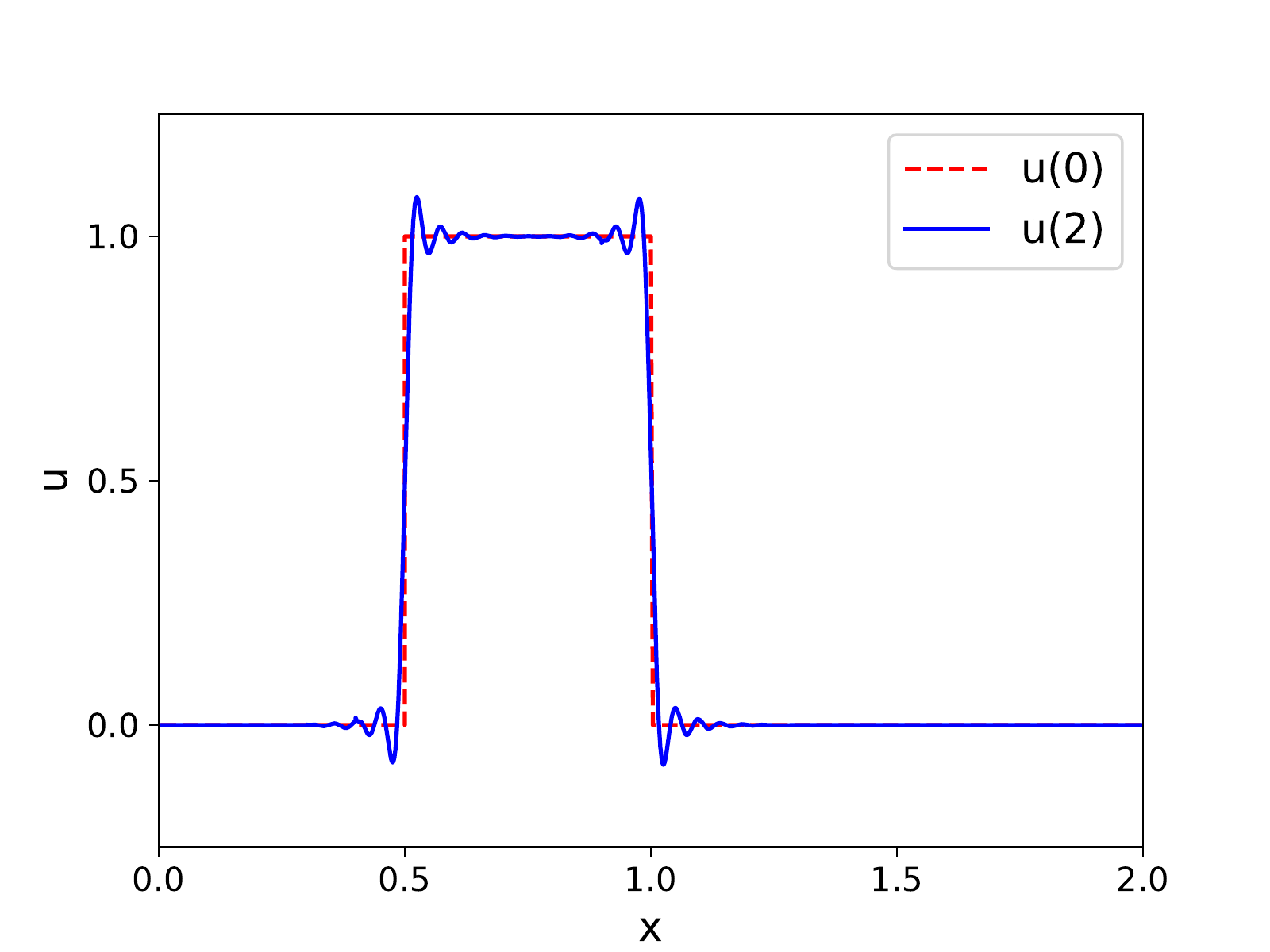}
    \caption{Numerical solution}
    \label{fig:linear_SSPRK33_sol}
  \end{subfigure}%
  ~
  \begin{subfigure}[b]{0.49\textwidth}
    \includegraphics[width=\textwidth]{%
      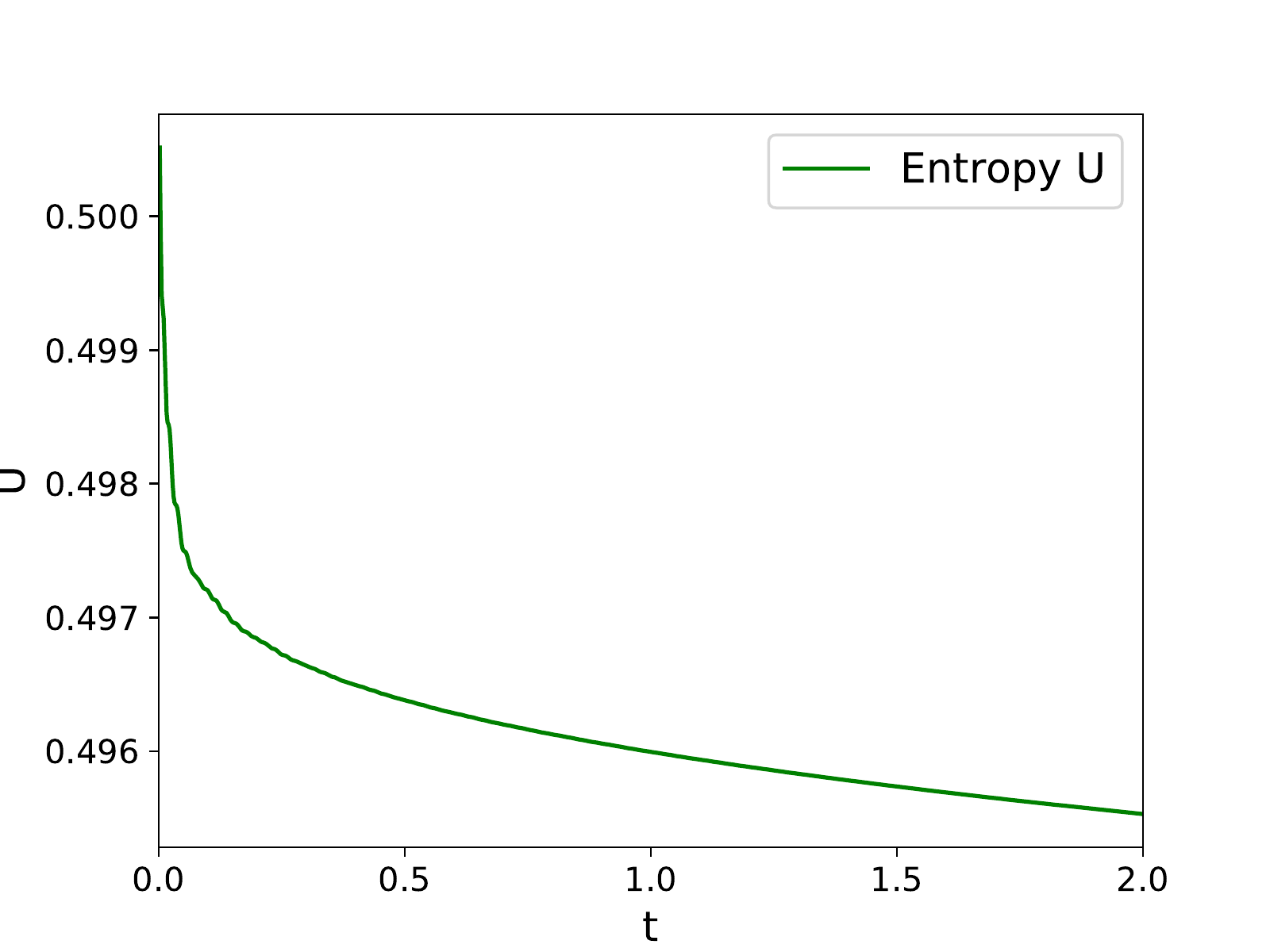}
    \caption{$L^2$ entropy $U$}
    \label{fig:linear_SSPRK33_ent}
  \end{subfigure}%
  \caption{Numerical solution and its entropy for linear advection using $N = 20$ elements with
           polynomials of degree $\leq p = 9$, $1 \ 000$ timesteps of the SSPRK(3,3) method, 
           and a full upwind numerical flux.
           }
  \label{fig:linear_SSPRK33}
\end{figure}

Further, a numerical full upwind flux $\fnum(u_-,u_+) = u_-$ and time 
integration by the third order SSPRK method using three stages - SSPRK(3,3) - by Gottlieb and Shu 
\cite{gottlieb1998total} has been used. 
It is known \cite{cockburn1989tvb} that for stability to hold the time step size $\Delta t$ should 
be bounded as $\Delta t \leq C \frac{h}{(2p+1) \lambda_{\text{max}}}$, where $\lambda_{\text{max}} 
= 1$ is the magnitude of the velocity and $h=1/I$ is the (local) mesh size.  
Using $1 \ 000$ time steps, the CFL number $C$ has been chosen as $C = 0.38$.

In \cite{kuzmin2006flux}, it is suggested that the linear advection equation is highly suited to test shock capturing schemes for several reasons:
\begin{enumerate}
    \item 
    The linear advection equation is the simplest partial differential equation that can 
feature discontinuous solutions. 
    Thus, the (shock capturing) method can be observed in a well-understood setting, isolated from nonlinear effects. 
    
    \item 
    The knowledge of the exact solution $u(t,x) = u_0(x-t)$, in particular, allows to suspend the 
    discussion of shock senors. 
    By eliminating potential shortcomings of a shock sensor, one is able to solely examine the behaviour of the shock capturing method. 
    
    \item 
    At the same time, the linear advection equation provides a most challenging example to be treated. 
    Similar to contact discontinuities in the Euler equations, discontinuities are not self-steeping. 
    Once such a discontinuity is smeared by the method, nothing will recover it to its original sharp shape.
    
\end{enumerate}

Furthermore, the $L^2$ entropy $U$ associated to \eqref{eq:linear_advection} remains constant for 
the exact solution.
While the entropy stable numerical flux introduces dissipation at the element boundaries, 
and thus in the spatial semidiscretization, the SSPRK method does in the time integration.
As a consequence, the entropy in Figure \ref{fig:linear_SSPRK33_ent} slightly 
decreases.
Both mechanisms alone, entropy stable (dissipative) numerical fluxes and time integration, have 
serious drawbacks however.
First, ($L^2$) stability of SSPRK schemes is only ensured when the simple forward Euler 
method $u^{n+1} = u^n + \Delta t \partial_t u^n$ already provides ($L^2$) stability 
\cite{gottlieb2001strong}. 
Yet, the forward Euler method is most often not stable, since
\begin{equation}
  \norm{u^{n+1}}^2 
    = \norm{u^n}^2 + 2 \Delta t \text{Re} \scp{u^n}{\partial_t u^n} + \left( \Delta t \right)^2 
      \norm{\partial_t u^n}^2, 
\end{equation}
where the second term might be estimated due to a stable semidisretisation, while the last term is 
non-negative and might render the forward Euler method unstable. 
See 
\cite{gottlieb1998total,levy1998semidiscrete,gottlieb2001strong,ketcheson2008highly,ranocha2016time} 
for more details on SSP time discretisation methods.
Second, and even more serious, are the shortcomings of shock capturing just by dissipative 
numerical fluxes. 
When dissipation is added just at the element boundaries, no further resolution can be 
obtained near shock discontinuities when the polynomial degree is increased. 
The same holds for increasing the number of elements, 
while keeping the polynomial degree fixed\footnote{which would also contradict the idea of high 
order methods}. 
Both cases are illustrated in Figure \ref{fig:linear_resolution}. 
Sub-cell resolution can neither be enhanced by increasing the polynomial degree nor by refining the mesh. 

\begin{figure}[!htb]
  \centering
  \begin{subfigure}[b]{0.49\textwidth}
    \includegraphics[width=\textwidth]{%
      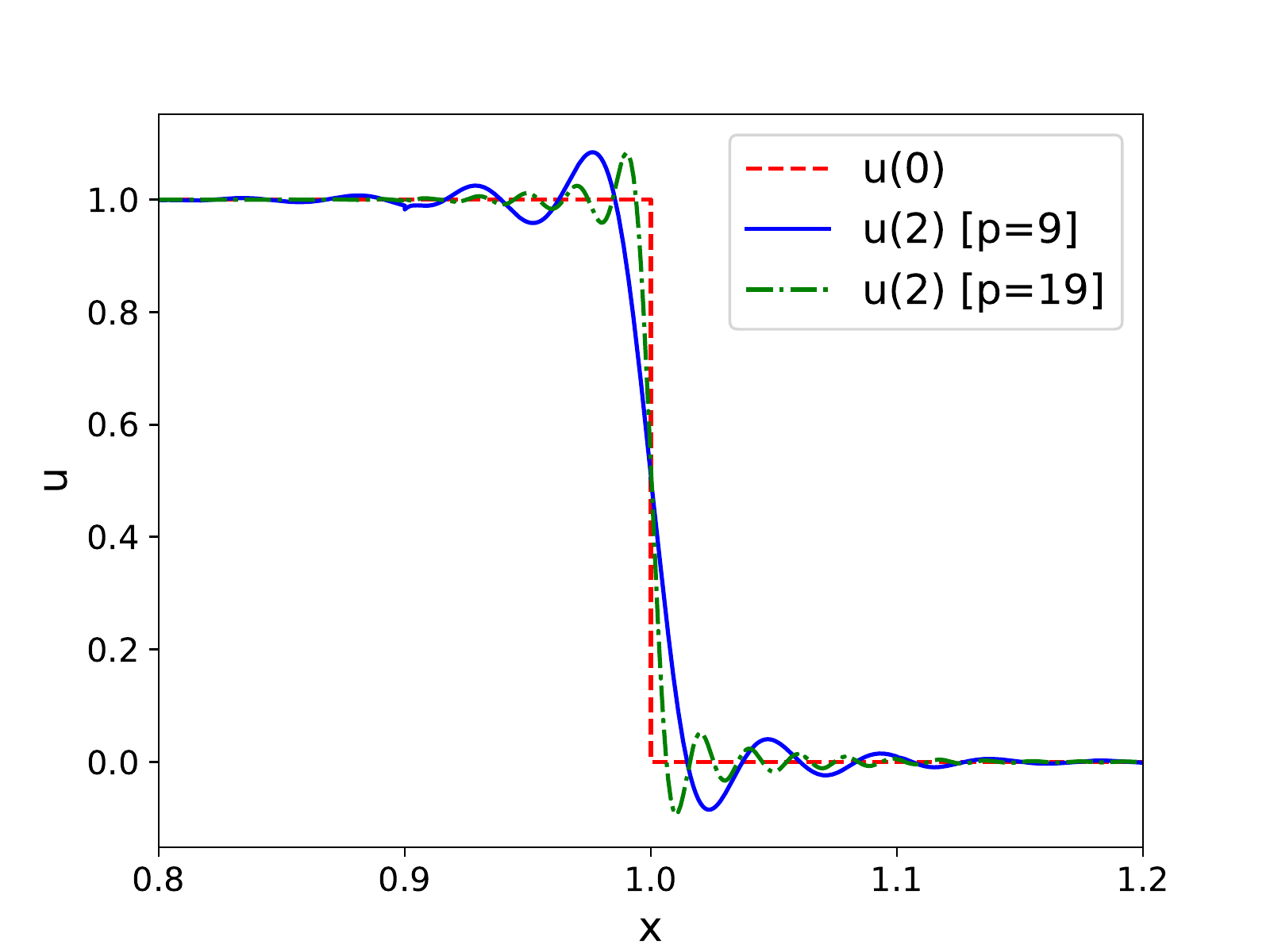}
    \caption{Solutions for polynomial degrees $p = 9,19$ and $N=20$ elements}
    \label{fig:linear_resolution_degree}
  \end{subfigure}%
  ~
  \begin{subfigure}[b]{0.49\textwidth}
    \includegraphics[width=\textwidth]{%
      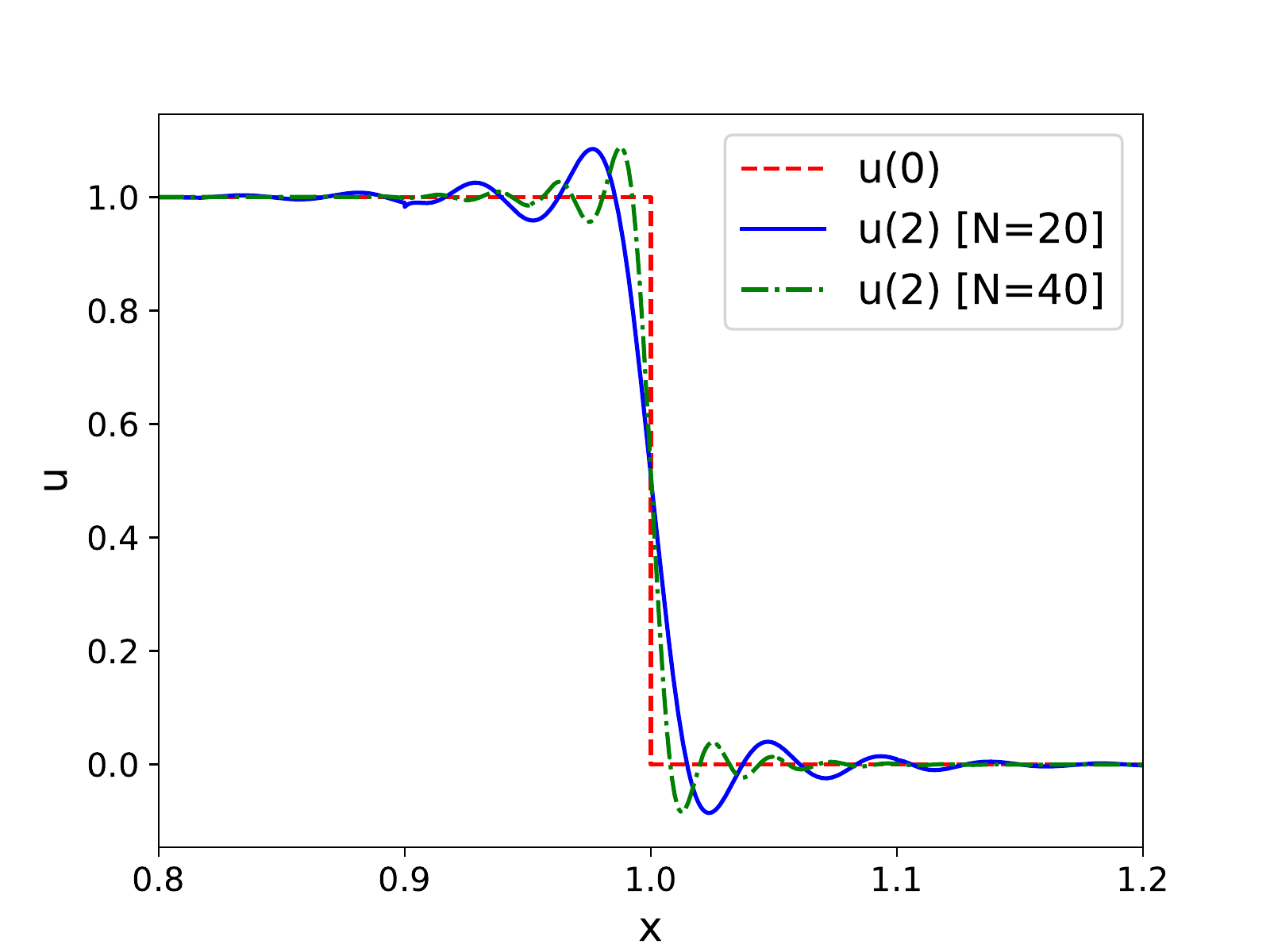}
    \caption{Solutions for polynomial degree $p=9$ and $N = 20,40$ elements}
    \label{fig:linear_resolution_elements}
  \end{subfigure}%
  \caption{Numerical solutions for linear advection using $10 \ 000$ timesteps.
           }
  \label{fig:linear_resolution}
\end{figure}

In both cases, the spurious oscillations are just closer to the discontinuity but will not vanish. 
A common approach therefore is to identify the elements lying in the shock region and to reduce the order of approximation there, 
see for instance \cite{baumann1999discontinuous,burbeau2001problem} and references therein.
Since this increases inter-element jumps and thus the amount of dissipation naturally added by the DG method, at latest for piecewise constant approximations, the method should be able to handle any shock. 
Yet, it should be stressed that decreasing the order of approximation is equivalent to adding dissipation proportional to $\mathcal{O}(\Delta x)$. 
Clearly, the accuracy will be reduced.
Thus, a widely accepted observation is that the solution, in fact, can be at most first order accurate near shocks. 
A common idea to bypass this problem is to adaptively refine the mesh in regions of discontinuity. 
Shocks, however, are lower dimensional objects and strongly anisotropic. 
An effective strategy for mesh adaptation therefore needs to incorporate some degree of directionality, especially in three dimensions. 
See for instance \cite{dervieux2003theoretical} and references therein.

In this work a more simple approach will be exploited. 
Sub-cell resolution will be enhanced by an improvement of the artificial viscosity method, 
\begin{equation}
    \partial_t u + \partial_x f(u) = 0 
    \quad \mapsto \quad  
    \partial_t u + \partial_x f(u) = \varepsilon \ \partial_{xx} u, 
    \quad \varepsilon > 0.
\end{equation}

\subsection{Vanishing and artificial viscosity}

Originally, the idea of artificial viscosity stems from the vanishing viscosity method to show existence of entropy solutions, see \cite{lax1973hyperbolic}. 
There, entropy solutions are constructed as $L^1$ limits of solutions $u_\varepsilon$ of the parabolic equation 
\begin{align}
    \partial_t u_\varepsilon + \partial_x f(u_\varepsilon) & = \varepsilon \ \partial_{xx} u_\varepsilon, 
    \quad \varepsilon > 0, \\ 
    u_\varepsilon(0,x) & = u_0(x),
\end{align}
for $\varepsilon \to 0$. 
Thus, another way to characterise physically realizable solutions is to identify them as limits of solutions of equations in which a small dissipative mechanism has been added. 

In their pioneering work \cite{vonneumann1950method}, von Neumann and Richtmyer revised this idea to construct stable FD schemes for the equations of hydrodynamics by including artificial viscosity terms. 
As they pointed out, when viscosity is taken into account, shocks are seen to be smeared out, so that the mathematical surfaces of discontinuity are replaced by thin layers in which pressure, density, temperature, etc. vary rapidly but continuously. 
The overall concept is to approximate (discontinuous) entropy solutions by smoother solutions of a 
parabolic equation and to apply the numerical method to this new equation, where shocks are 
now replaced by thin but continuous layers. 
Often, the smooth approximation $u_\varepsilon$ is also called a \emph{viscose profile} to the entropy solution $u$.

\subsection{A local artificial viscosity}

Considering a viscose profile over the whole domain $\Omega$, shocks might be spread over several elements if not even the whole domain. 
Furthermore, also other (especially small-scale) features of the original solution away from shocks typically get smeared by a global artificial viscosity, i.e. $\varepsilon = const$ on $\Omega$. 
Breaking new ground in \cite{persson2006sub}, Persson and Peraire therefore proposed a local artificial viscosity in the framework of DG methods. 
By locally adapting the viscosity strength $\varepsilon$, artificial viscosity is just added in the elements where shocks arise. 
Thus, the now piecewise constant function $\varepsilon$ is chosen $\varepsilon = 0$ away from discontinuities and $\varepsilon > 0$ in elements with shocks. 
Discontinuities that may appear in the original solution will spread over layers of thickness $\mathcal{O}(\varepsilon)$ in the solution of the modified equation.
Hence, Persson and Peraire more precisely suggest that $\varepsilon$ should be chosen as a function of the resolution available in the approximating space. 
Since this resolution given by a piecewise polynomial of order $p$ and element width $h$ scales like $\frac{h}{p}$, values $\varepsilon \in \mathcal{O}\left( \frac{h}{p} \right)$ are taken near shocks. 

While there are works of Barter and Darmofal \cite{barter2010shock} as well as Klöckner, Warburton, and Hesthaven \cite{klockner2011viscous} that emphasize certain drawbacks of a piecewise constant artificial viscosity, the present work is the first to decline this approach based on a strictly theoretical analysis. 
In \cite{barter2010shock}, the authors already observed oscillations in state gradients  which pollute the downstream flow for element-to-element variations in the artificial viscosity.  
In the following however, it is demonstrated that already the very first design principle \eqref{eq:conservation} of conservation fails for a just piecewise continuous artificial viscosity. 

Decomposing the domain $\Omega$ into disjoint, face-conforming elements $\Omega_i$, one has
\begin{equation}
    \od{}{t} \int_{\Omega_i} u 
        = - \int_{\Omega_i} \partial_x f(u) + \varepsilon_i \int_{\Omega_i} \partial_{xx} u 
        = - f(u)\Big|_{\partial \Omega_i} + \varepsilon_i \partial_x u \Big|_{\partial \Omega_i}
\end{equation}
in every element, since $\varepsilon$ is piecewise constant with value $\varepsilon_i$ on the $i$-th element.
Now putting the elements together, the rate of change of the total amount of $u$ is given by 
\begin{equation}
    \od{}{t} \int_{\Omega} u 
        = - \sum_{i} f(u)\Big|_{\partial \Omega_i} + \sum_{i} \varepsilon_i \partial_x u \Big|_{\partial \Omega_i}.
\end{equation}
Assuming a sufficiently smooth solution and therefore $u_- = u_+$ at the element boundaries, this reduces to
\begin{equation}
    \od{}{t} \int_{\Omega} u 
        = - f(u)\Big|_{\partial \Omega} + \sum_{i} \varepsilon_i \partial_x u \Big|_{\partial \Omega_i}.
\end{equation}
Note that especially for periodic boundary conditions one thus has 
\begin{equation}
    \od{}{t} \int_{\Omega} u 
        = \sum_{i} \varepsilon_i \partial_x u \Big|_{\partial \Omega_i},
\end{equation}
which typically is not equal to zero and thus violates the design principle \eqref{eq:design_con} of conservation. 
Demonstrated by this analysis, a just piecewise constant artificial viscosity admittedly has some critical drawbacks and thus should be rejected.

\subsection{Recent continuous refinements: Smoothing the artificial viscosity}
\label{sub:recent}

In \cite{barter2010shock} Barter and Darmofal took up the work \cite{persson2006sub} of Persson and Peraire and further improved their ideas. 
Clear numerical shortcomings of a piecewise constant artificial viscosity are stressed in their work. 
In particular, they note that the element-to-element variations induce oscillations in state gradients, which pollute the downstream flow. 
While numerical oscillations are damped, oscillations remain in the derivative $\partial_x u$. 
Even though they seem to miss the violation of conservation, they clearly point out the missing 
smoothness of the viscosity as the crucial problem. 
Thus, in \cite{barter2010shock} a smooth artificial viscosity was developed by employing an artificial viscosity PDE model which is appended to the system of governing equations. 
Effectively, their idea to enhance smoothness of the artificial viscosity is "diffusing the diffusivity". 

Without separately discussing the drawbacks of this particular approach (which will be stressed later), we now analyse the consequences of a 'sufficiently smooth' artificial viscosity. 
Working with the more general viscosity term $\left( \partial_x \varepsilon(x) \partial_x \right) u$ this time, 
an analogous analysis to the one before yields 
\begin{equation}
    \od{}{t} \int_{\Omega} u 
        = - f(u)\Big|_{\partial \Omega} + \sum_{i} \varepsilon(x) \partial_x u \Big|_{\partial \Omega_i} 
        = - f(u)\Big|_{\partial \Omega} + \varepsilon(x) \partial_x u \Big|_{\partial \Omega}.
\end{equation}
Again assuming a periodic behaviour, now also for $\varepsilon$, this time 
\begin{equation}
    \od{}{t} \int_{\Omega} u 
        = 0,
\end{equation}
arises, which in fact satisfies the design principle \eqref{eq:design_con} of conservation. 
Furthermore, the general design principle of conservation for arbitrary boundary conditions \eqref{eq:conservation} can be satisfied by enforcing a compactly supported viscosity $\varepsilon(x)$ on $\Omega$. 
By noting that in the above analysis, a continuous viscosity $\varepsilon(x)$ is already 
'sufficiently smooth', the following Theorem is immediately proven.

\begin{theorem}\label{thm:conservation}
    Augmenting a conservation law $\partial_t u + \partial_x f(u) = 0$ with a continuous and on $\Omega$ compactly supported artificial viscosity $\left( \partial_x \varepsilon(x) \partial_x \right) u$
    on the right hand side, preserves conservation. 
    I.e. 
    \begin{equation}
        \od{}{t} \int_{\Omega} u 
        = - f(u)\Big|_{\partial \Omega}
    \end{equation}
    holds for solutions of $\partial_t u + \partial_x f(u) = \left( \partial_x \varepsilon(x) \partial_x \right) u$.
\end{theorem}

In accordance to Theorem \ref{thm:conservation}, in \cite{klockner2011viscous} Klöckner et al. 
numerically observed that there seems to be no advantage in having viscosities $\varepsilon \in C^k$ 
for $k > 0$.\footnote{$\varepsilon \in C^k$ for $k > 0$ is required for higher order viscosities, 
for instance $\left(\partial_{xx} \nu \partial_{xx} \right) u$. Then, $\frac{\d^k 
\nu}{\d t^k} (\pm 1) = 0$ is further required for conservation.} 
Thus, they have formulated an algorithm to enforce continuity of the viscosity - as the authors think, more efficient than the one of Barter and Darmofal - by simple linear interpolation. 
Building up on a given piecewise constant viscosity, they propose the following steps:
\begin{enumerate}
    \item 
    At each vertex, collect the maximum viscosity occurring in each of the adjacent elements.
    
    \item 
    Propagate the resulting maximum back to each element adjoining vertex. 
    
    \item 
    Use a linear ($P^1$) interpolant to extend the values at the vertices into a viscosity on the entire element.
\end{enumerate}

While the above algorithm works perfectly fine in regards to enforce continuity of the artificial viscosity, it inherits a critical disadvantage, which is also shared by the approach of Barter and Damofal \cite{barter2010shock}. 
As Sheshadri and Jameson already stated in \cite{sheshadri2014shock}, 
enforcement of continuity of the artificial viscosity $\varepsilon$ across element boundaries  
increases the footprint of the added dissipation again. 
This is noticed immediately when consulting Figure \ref{fig:cont_av}, where the algorithm of Kl\"ockner et al. is illustrated for a simple example. 

\begin{figure}[!htb]
  \centering
  \begin{subfigure}[b]{0.49\textwidth}
    \includegraphics[width=\textwidth]{%
      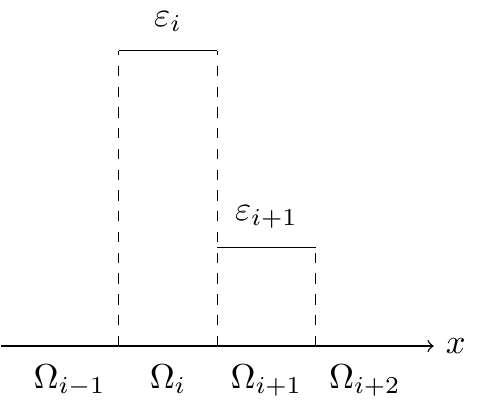}
    \caption{Piecewise constant artificial viscosity}
    \label{fig:piecewise}
  \end{subfigure}%
  ~
  \begin{subfigure}[b]{0.49\textwidth}
    \includegraphics[width=\textwidth]{%
      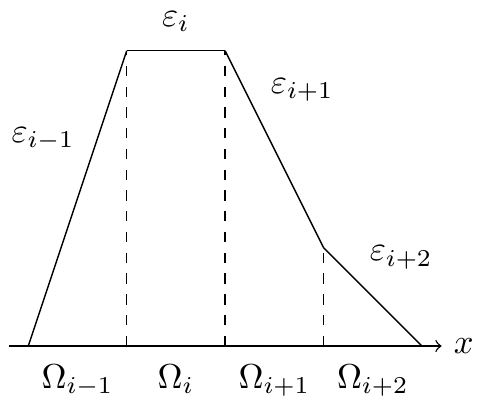}
    \caption{Continuous artificial viscosity}
    \label{fig:continuous}
  \end{subfigure}%
  \caption{A piecewise constant viscosity and the corresponding continuous viscosity 
constructed by the algorithm of Kl\"ockner et al..
           }
  \label{fig:cont_av}
\end{figure}

First introducing local artificial viscosity to resolve shock discontinuities within a single 
element, yet increasing the footprint again by enforcing continuity, the circle closes.

%% file: 5_modal.tex
\section{\texorpdfstring{Modal filtering and novel $C^\infty_0$ artificial viscosities}{Modal 
filtering and novel smooth and compactly supported artificial viscosities}}
\label{sec:modal}

When enforcing continuity of the viscosity, and therefore ensuring conservation, 
the initial localization of the artificial viscosity proposed by Persson and Peraire 
\cite{persson2006sub} is destroyed again. 
Shocks are spread over several cells and sub-cell resolution is prevented. 

Yet, another drawback arises from the artificial viscosity method. 
Augmenting the original conservation law with a parabolic term $\left( \partial_x \varepsilon \partial_x \right) u$ significantly decreases the efficiency of the numerical scheme. 
In addition to the typical time step restriction $\Delta t \sim \Delta x$ an additional restriction $\Delta t \sim \frac{\left( \Delta x \right)^2 }{ \| \varepsilon \|_\infty}$ of an explicit time step arises for the artificial viscosity extension $\partial_t u + \partial_x f(u) = \left( \partial_x \varepsilon \partial_x \right) u$. 
This second restriction is potentially more harsh and might reduce the global time step considerably. 
Note that also methods such as the local DG method of Cockburn and Shu \cite{cockburn1998local}, where the viscosity extension is rewritten as a system of first order hyperbolic equations 
\begin{equation}\label{eq:local-DG}
\begin{aligned}
    \partial_t u + \partial_x f - \partial_x q  & = 0, \\
    q - \varepsilon \partial_x u & = 0,
\end{aligned}
\end{equation}
by introducing the auxiliary flux variable $q$, just shift the problem. 
The efficiency still is reduced, now by the increased complexity of the system under consideration.

To sum up, there are two major problems which have to be addressed:
\begin{enumerate}
    \item 
    Spreading of shocks over several elements by enforcement of continuity.
    
    \item 
    Decreased efficiency of the method by introducing second (or higher) order terms of the artificial viscosity.
\end{enumerate}

The first problem will be overcome by introducing $C^\infty_0$ artificial viscosities, 
which are compactly supported on the corresponding element. 
A well-known representative of this class of viscosities is the so called Legendre artificial 
viscosity, $\partial_x \left( 1 - x^2 \right) \partial_x u $, which will be briefly revised 
in subsection \ref{subsec:Legendre}. 
This viscosity is often used since it corresponds to certain modal exponential filters. 
Using these modal filters instead of the original artificial viscosity method, the second 
problem of additional time step restrictions can also be overcome. 
Yet, in section \ref{subsec:Legendre} the Legendre artificial viscosity, and therefore modal filtering in the Legendre basis, is declined due to serious shortcomings of the corresponding viscosity function. 
To the best of the authors' knowledge, this work is the first to address this particular 
problem of the Legendre artificial viscosity and corresponding modal filters. 
By an argument of Bochner (1929), the same modal filters will even be rejected for all 
orthogonal bases of Jacobi polynomials. 

In order to fill the void that results from rejecting the state of the art viscosities, 
novel $C^\infty_0$ artificial viscosities are proposed in subsection 
\ref{subsec:novel_viscosities}. 
To the best of the authors knowledge, these viscosities are proposed for the first time.

\subsection{Legendre artificial viscosity and modal filtering}
\label{subsec:Legendre}

A commonly utilised approach to overcome Problem 1, i.e. an increased footprint of the viscosity 
caused by enforcement of continuity, is to use the \emph{Legendre viscosity} defined on the 
reference element $[-1,1]$ by 
\begin{equation}\label{eq:Legendre_AV}
    \varepsilon(x) = \varepsilon \left( 1 - x^2 \right),
\end{equation}
where $\varepsilon \in \R_0^+$ is called the \emph{viscosity strength} and $\nu(x) =\left( 1 - x^2 \right)$ is referred to as the \emph{viscosity distribution}.

Using the Legendre artificial viscosity, shocks are again resolved within a single element, 
where now conservation is satisfied by Theorem \ref{thm:conservation} and entropy stability is 
further ensured by the subsequent Theorem \ref{thm:stability}. 
The Legendre viscosity is such a widespread tool because it enables one to bypass Problem 2 
in a certain sense. 
Here, a procedure first proposed by Majda, McDonough, and Osher in \cite{majda1978fourier} as well 
as by Kreiss and Oliger in \cite{kreiss1979stability} is utilised. 
Also see the monograph \cite{gottlieb2001spectral}.

Applying a first order operator splitting in time, solving the artificial viscosity 
extension 
$\partial_t u + \partial_x f(u) = \varepsilon \left( \partial_x \nu \partial_x \right)u$ 
is divided into two steps: 
\begin{align}
    & \text{(1st) solve } \quad \partial_t u + \partial_x f(u) = 0, \label{eq:1st} \\ 
    & \text{(2nd) solve } \quad \partial_t u =  \varepsilon \left( \partial_x \nu \partial_x \right)u. \label{eq:2nd}
\end{align} 
Applying a numerical solver to this approach, going forward in time is done by alternately integrating \eqref{eq:1st} and \eqref{eq:2nd}. 
This can for instance be done by an explicit Runge-Kutta method. 
Note that in this approach additional time step restrictions are still present for the parabolic equation \eqref{eq:2nd}. 
When adjusting the artificial viscosity to the modal basis in which the numerical solution is expressed, however, equation \eqref{eq:2nd} can be solved exactly. 
The exact solution is obtained by applying a modal exponential filter to the numerical solution of 
equation \eqref{eq:1st} then.

The Legendre viscosity, in particular, is used so often, since the orthogonal Legendre polynomials $\{P_k\}$ arise as solutions of the corresponding Legendre differential equation 
\begin{equation}\label{eq:Legendre}
    \od{}{x} \left( 1 - x^2 \right) \od{}{x} P_k 
        = - \lambda_k P_k 
\end{equation}
of a Sturm-Liouville type for eigenvalues $-\lambda_k = -k(k+1)$. 
Thus, expressing the numerical solution $u_p$ with respect to the commonly used orthogonal basis of Legendre polynomials, $u_p(t,x) = \sum_{k=0}^p \hat{u}_k(t) P_k(x)$, equation \eqref{eq:2nd} reads 
\begin{equation} \small
    \sum_{k=0}^p \left(\partial_t \hat{u}_k \right)(t) \ P_k(x) 
        = \varepsilon \sum_{k=0}^p \hat{u}_k(t) \left[ \partial_x (1-x^2) \partial_x \right] P_k(x) 
        = - \varepsilon \sum_{k=0}^p \lambda_k \hat{u}_k(t) P_k(x).
\end{equation}
Thereby, the last equation follows from the eigenvalue equation \eqref{eq:Legendre} for the Legendre polynomials. 
A simple comparison of the time dependent coefficients results in a system of $p+1$ decoupled ODEs 
\begin{equation}
    \od{\hat{u}_k}{t} (t) = - \varepsilon \lambda_k \hat{u}_k(t), 
    \quad k=0,1,\dots,p,
\end{equation}
with solutions given by 
\begin{equation}
    \hat{u}_k (t) = C_k \cdot e^{- \varepsilon \lambda_k t}, \quad C_k \in \R.
\end{equation}
After one time step $\Delta t = t^{n+1} - t^n$, the coefficients $\hat{u}_k (t^{n+1})$ are 
hence given by 
\begin{equation}
    \hat{u}_k (t^{n+1}) 
        = C_k \cdot e^{- \varepsilon \lambda_k (\Delta t + t^n)} 
        = e^{- \varepsilon \lambda_k \Delta t} \cdot \hat{u}_k (t^{n}).
\end{equation}
Thus, solving the parabolic equation \eqref{eq:2nd} for the numerical solution $u_p$ is equivalent to multiplying the coefficients $\hat{u}_k$ with the function 
$\sigma(k) = e^{-\varepsilon \Delta t \lambda_k}$.
In order to speak of a modal filter for $\sigma(k) = \exp( -\varepsilon \Delta t \lambda_k )$, the exponent has to be rewritten as 
\begin{equation}\label{eq:exp_filter}
    \sigma \left( \frac{k}{p} \right) 
        = e^{ - \varepsilon p^2 \Delta t \left( \frac{k}{p} \right) \left( \frac{k+1}{p} \right)} 
        \approx e^{ - \varepsilon p^2 \Delta t \left( \frac{k}{p} \right)^2}.
\end{equation}
Only now, $\sigma: [0,1] \to [0,1]$ can be seen as an exponential filter of \emph{order} $2$ with \emph{filter strength} $\alpha := \varepsilon p^2 \Delta t$. 
See for instance the work \cite{hesthaven2008filtering} of Hesthaven and Kirby, as well as references therein.

By now, modal filtering is an established shock-capturing tool and was applied in a great number of 
works, 
\cite{hu1996absorbing,yang1997spectral,gottlieb2001spectral,boyd2001chebyshev,giraldo2002nodal,%
kanevsky2006idempotent,hesthaven2008filtering,meister2012application,meister2013extended,%
glaubitz2016application,offner2017stability}.
In our numerical tests however, we observed the Legendre artificial viscosity as well as the corresponding modal filter applied in a Legendre basis to perform poorer than for instance the artificial viscosity method of Kl\"ockner et al. 
In Figure \ref{fig:dismiss_Legendre}, this is demonstrated for the linear advection equation. 

\begin{figure}[!htb]
  \centering
  \begin{subfigure}[b]{0.49\textwidth}
    \includegraphics[width=\textwidth]{%
      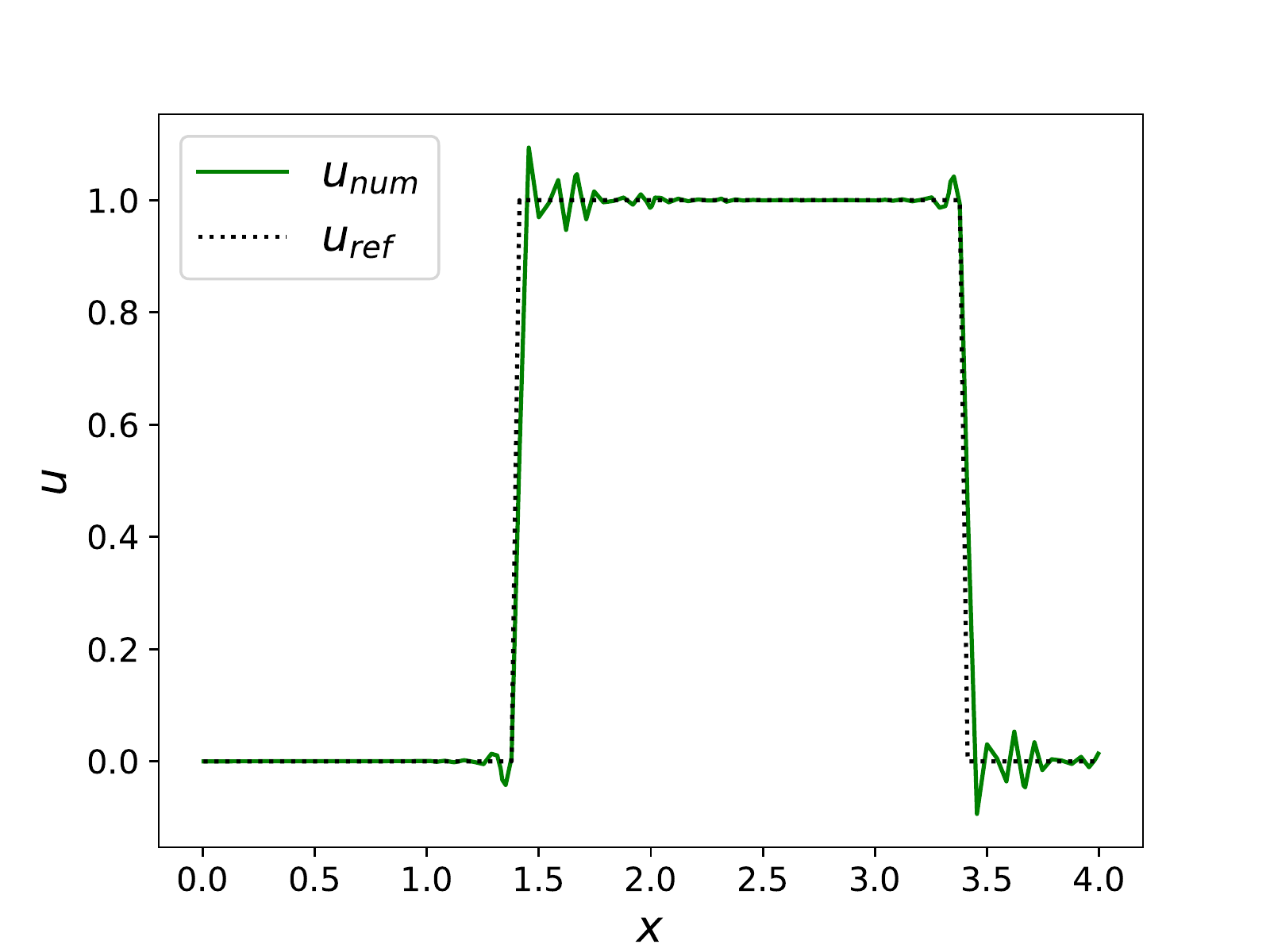}
    \caption{No viscosity}
    \label{fig:sol_no_vsic}
  \end{subfigure}%
  ~
  \begin{subfigure}[b]{0.49\textwidth}
    \includegraphics[width=\textwidth]{%
      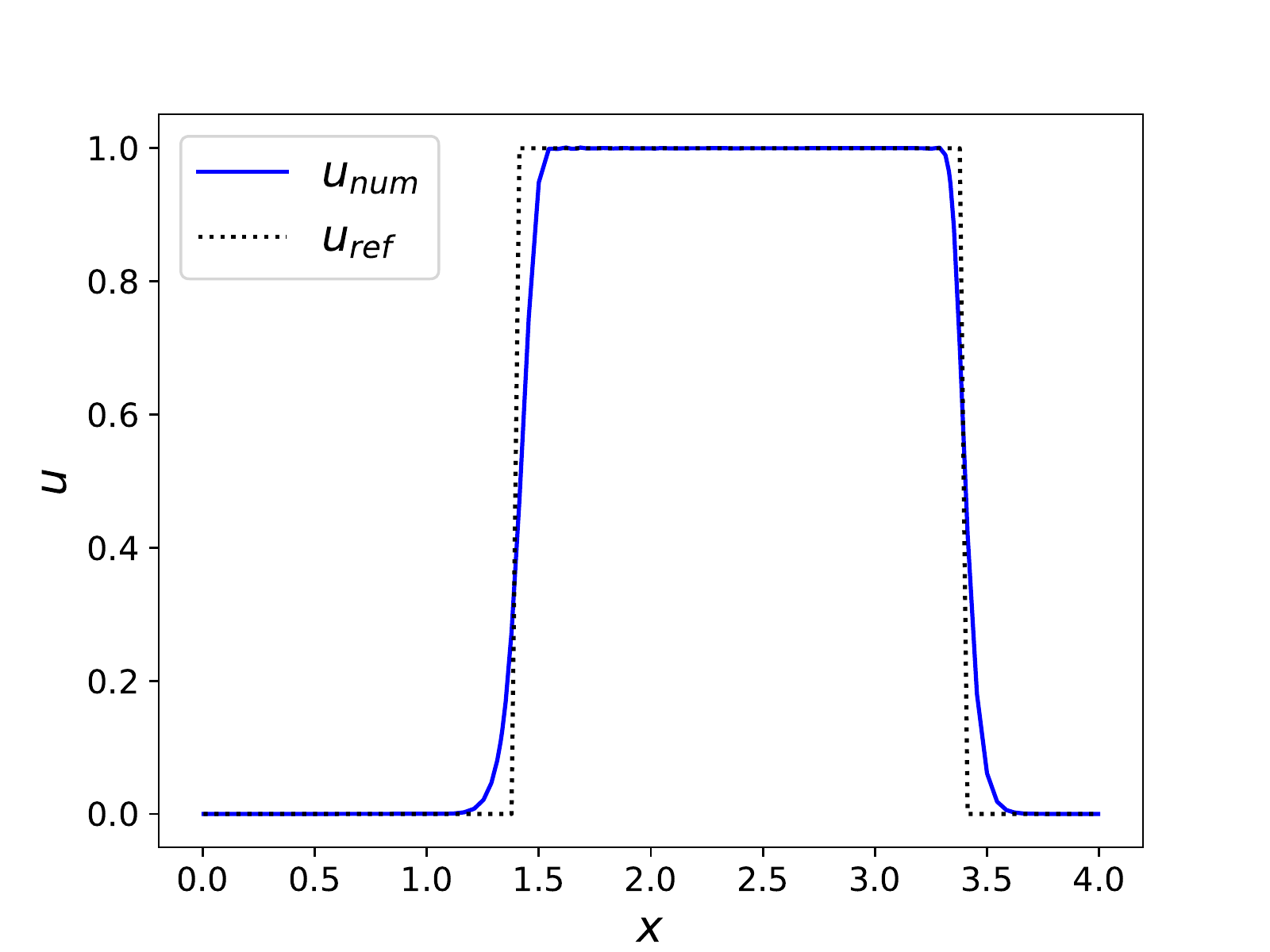}
    \caption{$C^0$ artificial viscosity}
    \label{fig:sol_Kloeckner}
  \end{subfigure}%
  \\
   \begin{subfigure}[b]{0.49\textwidth}
    \includegraphics[width=\textwidth]{%
      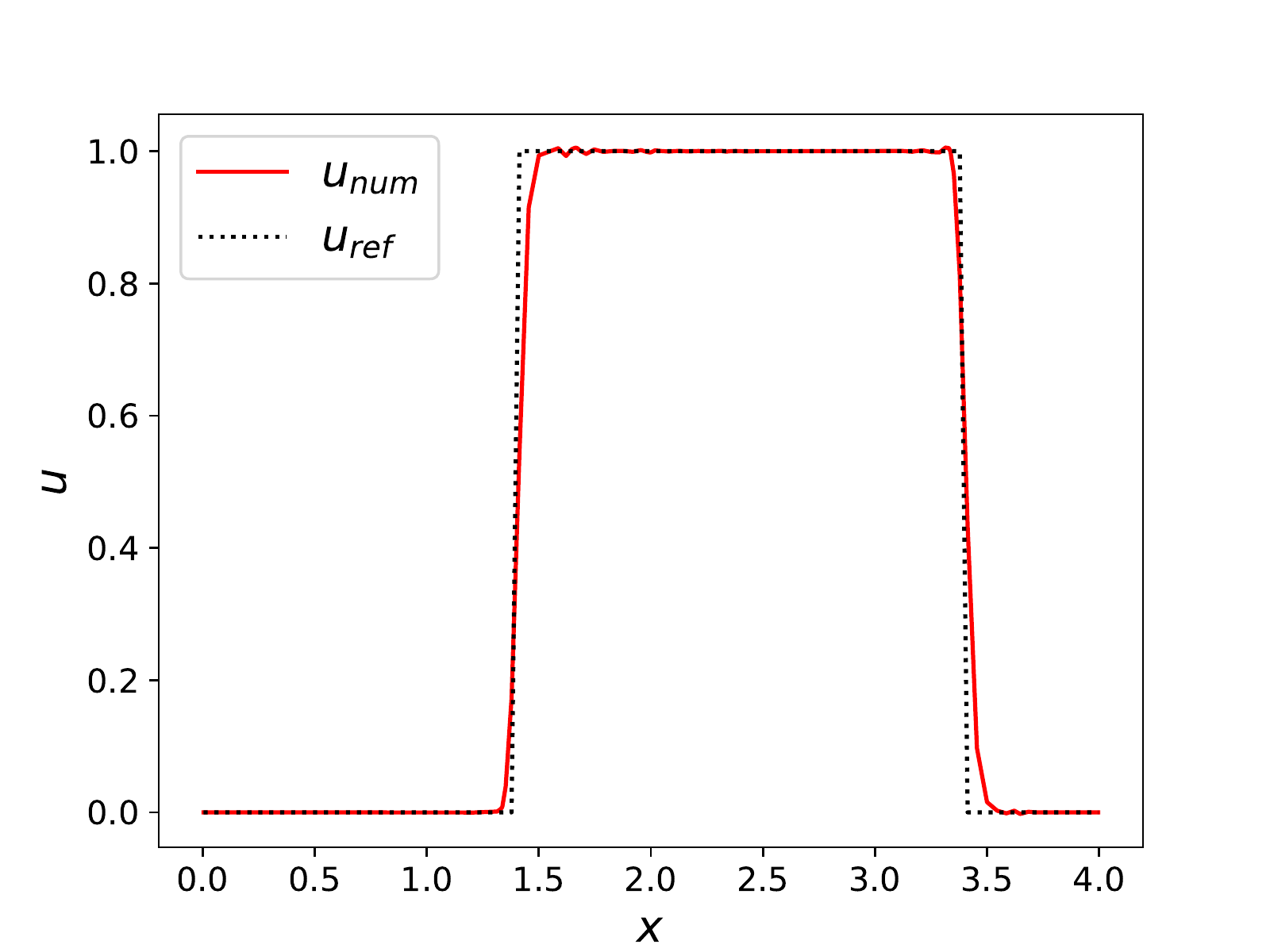}
    \caption{Legendre artificial viscosity}
    \label{fig:sol_Legendre}
  \end{subfigure}%
  ~
  \begin{subfigure}[b]{0.49\textwidth}
    \includegraphics[width=\textwidth]{%
      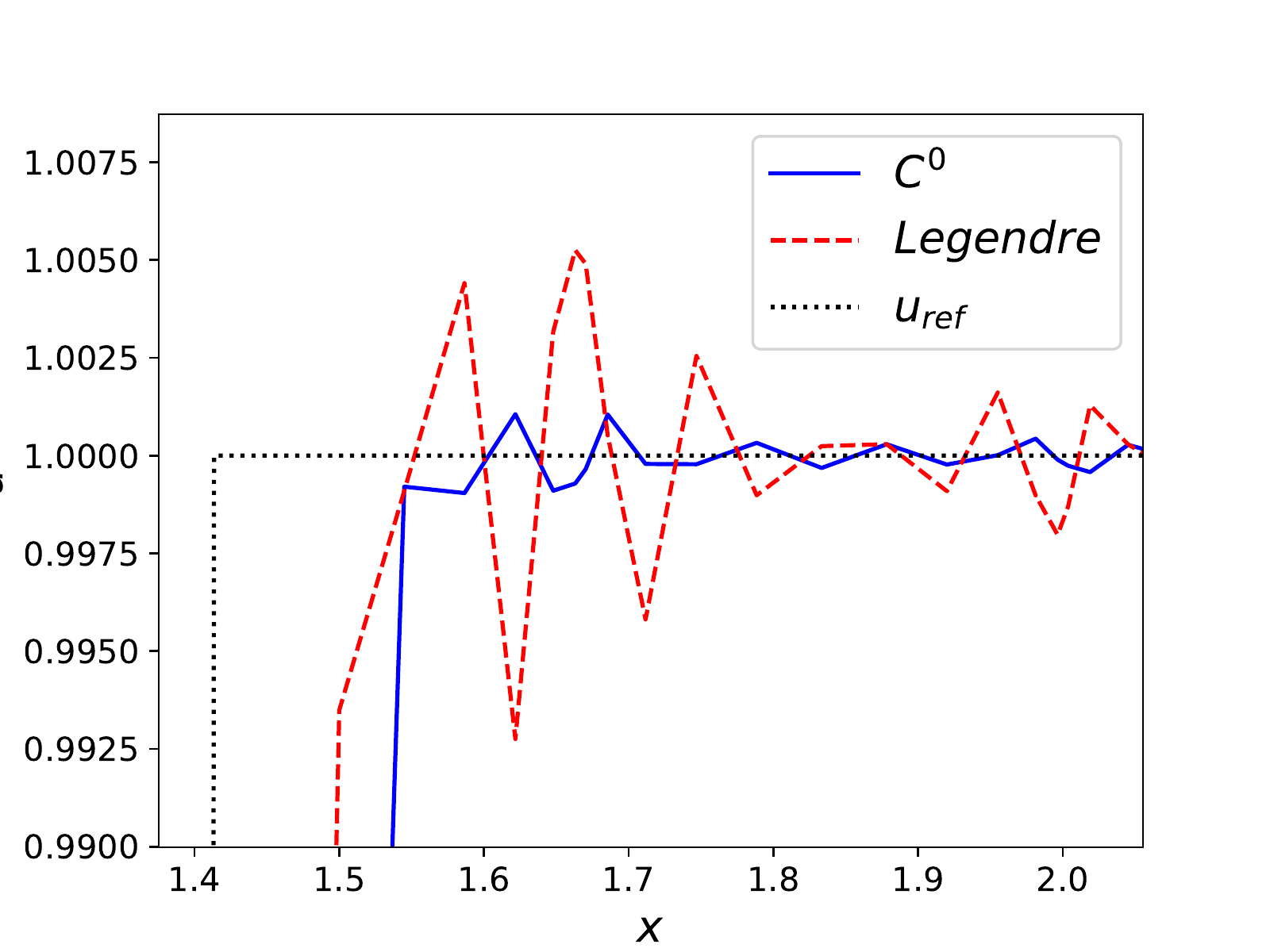}
    \caption{$C^0$ and Legendre AV - zoom}
    \label{fig:sol_zoom}
  \end{subfigure}%
  \caption{Numerical solutions and viscosities for the linear advection equation using $12$ elements and polynomials of degree $10$ or less.
           }
  \label{fig:dismiss_Legendre}
\end{figure}

While there are nearly no oscillations present anymore in the numerical solution obtained by the 
$C^0$ artificial viscosity of Kl\"ockner et al. (\ref{fig:sol_Kloeckner}), the numerical solution 
obtained by the Legendre artificial viscosity still shows spurious oscillations 
(\ref{fig:sol_Legendre}). 
This is illustrated in greater detail in Figure \ref{fig:sol_zoom}, where the numerical solutions 
for 
the $C^0$ and Legendre artificial viscosity method are compared to the right of the jump 
discontinuity. 
By running various numerical tests, the particular viscosity distribution of the Legendre artificial viscosity was explored to be the determining factor for the spurious oscillations in the corresponding numerical solution. 
Further illustrated by the subsequent Figure \ref{fig:Gegenbauer_vis}, the Legendre viscosity rapidly vanishes away from the element center. 
Thus, if a shock discontinuity performs near a cell boundary, nearly no - in fact arbitrary little - dissipation is added there by the Legendre viscosity. 
This observation can be made whenever shocks perform near element boundaries and yield to the conclusion that the Legendre artificial viscosity method as well as corresponding modal filters should be rejected. 

Naturally, the question arises if this bottleneck might be overcome by utilising more 
suitable viscosities, which then again should correspond to modal filters in a proper basis of 
orthogonal polynomials. 
While the first attempt to construct more appropriate viscosity distributions will be tackled in subsection \ref{subsec:novel_viscosities} by constructing novel viscosity distributions, the possibility of still finding corresponding modal filters can immediately be rejected. 
The decisive property of the Legendre artificial viscosity is that the orthogonal Legendre 
polynomials occur as eigenfunctions of the Sturm-Liouville operator $Ly = \left( (1 - x^2) y' 
\right)'$ with eigenvalues $-\lambda_k = -k(k+1)$, i.e. 
\begin{equation}\label{eq:S-L_2nd_order}
    a(x) y'' + b(x) y' + c(x) y = \lambda y
\end{equation}
with $a(x) = 1 - x^2, b(x) = 2x$ and $c(x) = 0$ for the Legendre artificial viscosity.
Yet, following a simple argument of Bochner \cite{bochner1929sturm,routh1884some}, for every 
Sturm-Liouville operator featuring polynomials as eigenfunctions, the corresponding coefficients 
$a,b$ and $c$ must be polynomials of degree $2,1$ and $0$, respectively. 
Further keeping in mind that the resulting viscosity needs to vanish at the boundaries to ensure conservation and should be positive (see Theorem \ref{thm:stability}), up to a positive constant, $a(x) = 1-x^2$ is the very only choice. 
Note that every modal filter $\sigma_k = \exp( - \alpha k(k+1+\alpha+\beta) )$ applied 
in the associated basis of orthogonal Jacobi polynomials $\{ P_k^{(\alpha,\beta)} \}$ is equivalent 
to artificial viscosity of the form 
\begin{equation}\label{eq:Jacobi}
    \partial_x \left( 1 - x^2 \right) \partial_x u + \left( \beta - \alpha - ( \alpha + \beta ) x \right) \partial_x u,
\end{equation}
by a first order operator splitting in time, since the Jacobi polynomials are eigenfunctions of the more general Sturm-Liouville operator $Ly = \left( (1 - x^2) y' \right)' + \left( \beta - \alpha - ( \alpha + \beta ) x \right) y'$. 
While the dissipative term $\partial_x \left( 1 - x^2 \right) \partial_x$ is still the same, in this case also undesired dispersion would be added to the equation.

\subsection{\texorpdfstring{Novel $C^\infty_0$ artificial viscosities}{Novel smooth and compactly 
supported artificial viscosities}}
\label{subsec:novel_viscosities}

To the best of the authors' knowledge, this work is the first to decline the whole class of 
Jacobi viscosities and their corresponding modal filters. 
In order to somewhat fill the resulting void in sub-cell shock capturing methods for spectral schemes, novel viscosity distributions will be proposed in this subsection in order to overcome the shortcomings of the Legendre and Jacobi ones. 

Our idea is to construct more general smooth viscosities  
\begin{equation}\label{eq:general_vis}
    \varepsilon(x) = \varepsilon \ \nu(x),
\end{equation}
which still vanish at the element boundaries and have their peak at the element center.  
Conservation of the resulting \emph{$C^\infty_0$ artificial viscosities} $\varepsilon \left( \partial_x \nu(x) \partial_x \right) u $ is again ensured by Theorem \ref{thm:conservation}. 
The crucial requirement is for the viscosity distribution $\nu(x)$ to vanish at the element boundaries. 
Thus, other viscosity distributions than the Legendre one are obviously possible. 
Yet, non-negativity of these is mandatory to carry the design criteria of entropy stability \eqref{eq:stability} over to the artificial viscosity extension 
\begin{equation}\label{eq:visc_extension}
    \partial_t u + \partial_x f(u) = \varepsilon \left( \partial_x \nu(x) \partial_x \right) u.
\end{equation}
This is further summarized in 

\begin{theorem}\label{thm:stability}
    Augmenting a conservation law $\partial_t u + \partial_x f(u) = 0$ with a $C^\infty_0$ artificial viscosity $\varepsilon \left( \partial_x \nu(x) \partial_x \right) u $ with positive viscosity strength $\varepsilon$ and viscosity distribution $\nu \geq 0$, again results in entropy stability. 
    I.e. 
    \begin{equation}
        \od{}{t} \int_{\Omega} U 
        \leq - F(u)\Big|_{\partial \Omega}
    \end{equation}
    holds for solutions of $\partial_t u + \partial_x f(u) = \varepsilon \left( \partial_x \nu(x) \partial_x \right) u $ and an associated entropy-entropy flux pair $(U,F)$.
\end{theorem} 

\begin{proof}
    Multiplying the viscosity extension \eqref{eq:visc_extension} by an entropy gradient $U'$, i.e. satisfying $U'(u) f'(u) = F'(u)$, one obtains an entropy equation
    \begin{equation}
        \partial_t U + \partial_x F = \varepsilon U' \left( \partial_x \nu(x) \partial_x \right) u.
    \end{equation}
    Since $\nu(x) \geq 0$ and $U$ is convex, the inequality 
    \begin{align}
        \left( \partial_x \nu(x) \partial_x \right) U 
            & = \partial_x \left( \nu U' \partial_x u \right) \\ 
            & = \nu'(x) U' \partial_x u 
                + \nu(x) U'' \left( \partial_x u \right)^2 
                + \nu(x) U' \partial_{xx} u \\
            & \geq \nu'(x) U' \partial_x u 
                + \nu(x) U' \partial_{xx} u \\ 
            & = U' \left( \partial_x \nu(x) \partial_x \right) u
    \end{align}
    holds, and thus the entropy inequality 
    \begin{equation}
        \partial_t U + \partial_x F \leq \left( \partial_x \nu(x) \partial_x \right) U
    \end{equation}
    follows. 
    Finally, for the rate of change of entropy in the element $\Omega_i$ under consideration, 
    \begin{equation}
    \begin{aligned}
        \od{}{t} \int_{\Omega_i} U 
            & \leq - \int_{\Omega_i} \partial_x F(u) + \varepsilon \int_{\Omega_i} \left( 
\partial_x \nu(x) \partial_x \right) U \\
            & = - F(u)\Big|_{\partial \Omega_i} + \varepsilon \underbrace{\nu(x) \partial_x U 
\Big|_{\partial \Omega_i}}_{= 0}
            = - F(u)\Big|_{\partial \Omega_i}
    \end{aligned} 
    \end{equation}
    is proven, which ensures the design criteria of entropy stability \eqref{eq:stability}. 
\end{proof} 

Note that both, the local artificial viscosity of Persson and Peraire \cite{persson2006sub} 
using a piecewise constant viscosity as well as the continuous refinement and Klöckner et al. 
\cite{klockner2011viscous} yield entropy stability. 
In fact, the proof of Theorem \ref{thm:stability} easily adapts to continuous artificial 
viscosities over several cells as long as they are compactly supported on $\Omega$. 
This is furthermore in accordance with the well-known effect of dissipative mechanisms on shocks, such as heat conduction and more general viscosity terms \cite{vonneumann1950method}. 

\begin{remark}
    In section \ref{sec:motivation}, $C^0$ has been referred to the viscosity being globally 
$C^0$, i.e. continuous on the whole domain $\Omega$. 
    In the present section, however, $C^\infty_0$ refers to the viscosity being locally 
$C^\infty_0$, i.e. smooth and compactly supported in every element $\Omega_i$. 
    Globally, the viscosity will still be only continuous, if no additional restrictions at the 
boundaries are enforced for the viscosity.
\end{remark}

Finally, to overcome the drawback of the Legendre viscosity, novel viscosity distributions $\nu$ are proposed with $\nu \approx 1$ over more of the reference element $[-1,1]$:
\begin{enumerate}
    \item 
    The \emph{Gegenbauer viscosity} 
    \begin{equation}
        \nu(x) = \left( 1 - x^2 \right)^\lambda 
    \end{equation}
    for $\lambda > 0$ is somewhat the 'natural' generalisation of the Legendre viscosity. 
    Choosing small values for $\lambda$, the Gegenbauer viscosity becomes $\nu \approx 1$ over most of the interval, therefore fulfilling our devise for the novel viscosities. 
    This is further illustrated in Figure \ref{fig:Gegenbauer_vis} for $\lambda=1$ and $\lambda=1/10$.
    
    \item 
    The \emph{super Gaussian viscosity} 
    \begin{equation}\label{eq:super_gaussian}
        \nu(x) = e^{-\alpha x^{2\lambda}},
    \end{equation}
    where $\alpha = - \ln \varepsilon_M$ with $\varepsilon_M$ representing machine precision, 
    
    \item 
    as well as the \emph{Gevrey viscosity} 
    \begin{equation}
        \nu(x) = 
        \begin{cases}
            \exp\left( \frac{x^2}{\lambda (x^2 - 1)} \right) , & \quad 0 \leq |x| < 1, \\ 
            0 , & \quad |x| \geq 1,
        \end{cases},
    \end{equation}
\end{enumerate}
are inspired by corresponding weight functions, which were for instance studied in the context of Fourier (pseudo-)spectral reprojection \cite{gelb2006robust} as well as spectral mollifiers \cite{tadmor2002adaptive,tanner2006optimal}. 
Yet, this is the first time, they are proposed as viscosity distributions in an artificial viscosity method. 
Examples for both viscosities are illustrated in Figure \ref{fig:Gaussian_Gevry_vis}.

\begin{figure}[!htb]
  \centering
  \begin{subfigure}[b]{0.49\textwidth}
    \includegraphics[width=\textwidth]{%
      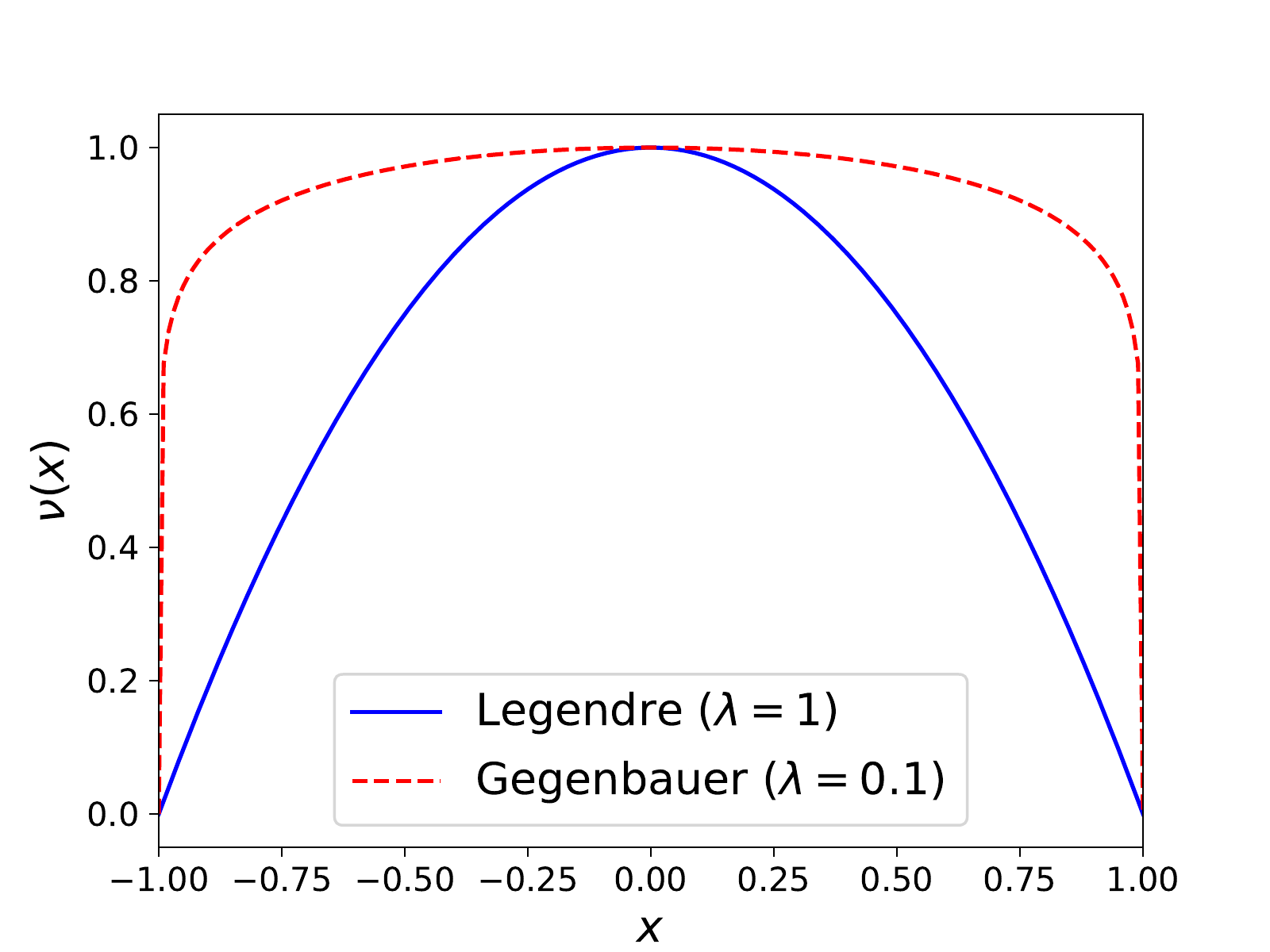}
    \caption{Legendre viscosity ($\lambda = 1$) and Gegenbauer viscosity for $\lambda = 0.1$}
    \label{fig:Gegenbauer_vis}
  \end{subfigure}%
  ~
  \begin{subfigure}[b]{0.49\textwidth}
    \includegraphics[width=\textwidth]{%
      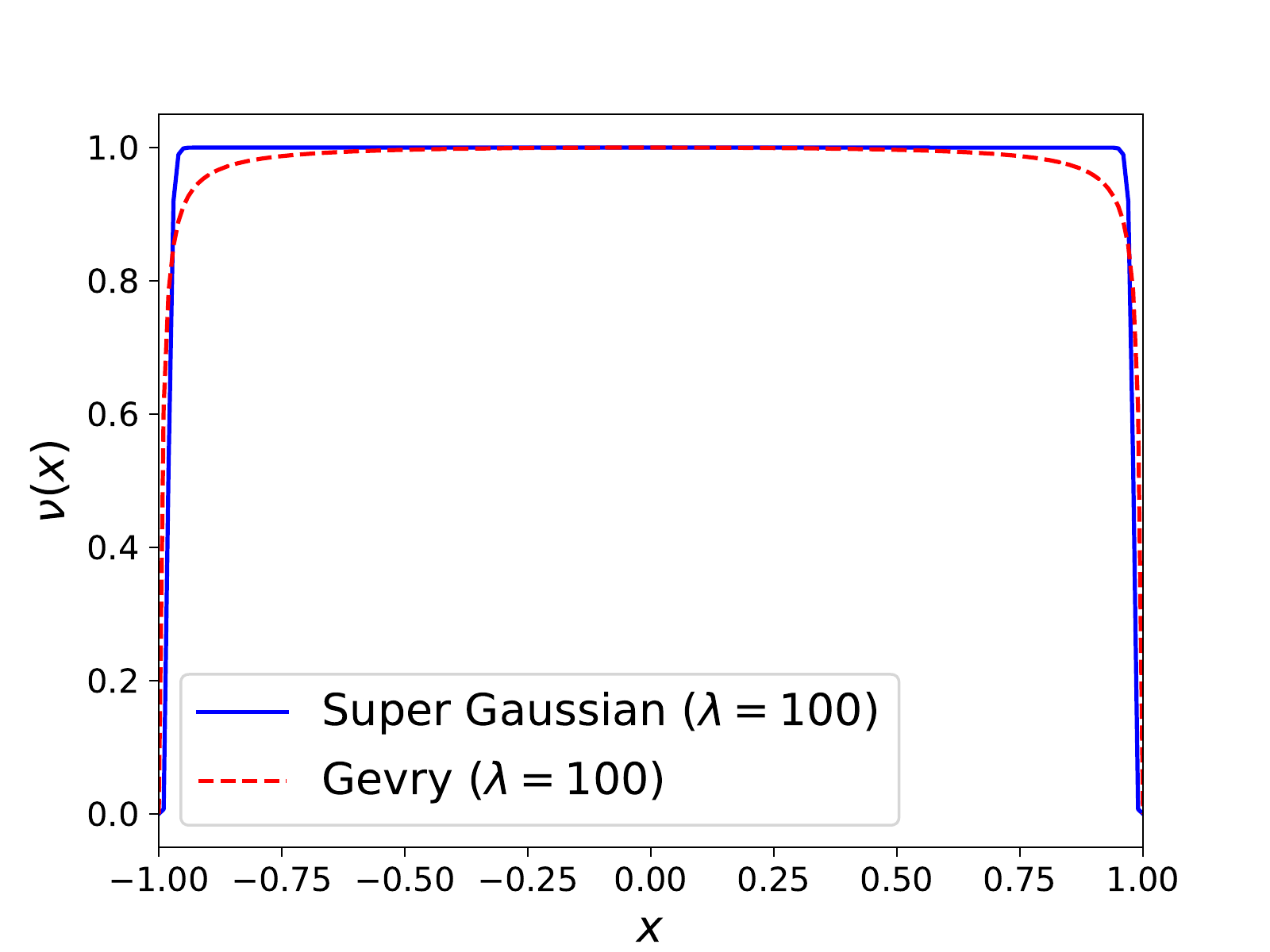}
    \caption{Super Gaussian viscosity and Gevry viscosity for 
$\lambda = 100$}
    \label{fig:Gaussian_Gevry_vis}
  \end{subfigure}%
  \caption{Legendre, Gegenbauer, super Gaussian, and Gevry viscosities.
           }
  \label{fig:viscosities}
\end{figure}

\begin{remark}
    It should be stressed that in principle $C_0^\infty$ viscosities yield the same time step 
restrictions as common viscosities. 
    While this drawback might be overcome by modal filtering, it was shown in section 
\ref{subsec:Legendre} that modal filtering corresponds to a Legendre viscosity, which in general is 
inappropriately distributed and was demonstrated in Figure \ref{fig:dismiss_Legendre} to 
result in undesired numerical solutions. 
    Yet, the compact nature of the new $C_0^\infty$ viscosities allow a more efficient treatment 
when, for instance, implicit-explicit (IMEX) time integration schemes \cite{ascher1995implicit} are 
applied to bypass too harsh (explicit) time step restrictions, since they are applied to less 
elements.
\end{remark}

\begin{remark}
    Implementation of the $C_0^\infty$ viscosity \eqref{eq:general_vis} is done in the very same way 
as for usual viscosities, i.e. by a local DG method \eqref{eq:local-DG}, see 
\cite{cockburn1998local}, or by direct calculation of second derivatives as done in 
\cite{ranocha2018stability}. 
  Thus, no additional code is necessary when going over from usual viscosities to 
$C_0^\infty$ viscosities. 
  In fact, $C_0^\infty$ viscosities are easier to implement from scratch, since no additional 
'smoothing procedure' \cite{barter2010shock,klockner2011viscous} as described in section 
\ref{sub:recent} is needed. 
  In our implementation, we have used the local DG method.
\end{remark}

Last but not least, the question of how to choose the viscosity strength $\varepsilon$ 
has to be addressed. 
It is evident that the strength should again be strictly greater than $0$ in elements where 
a shock is present and equal to $0$ everywhere else. 
The next section thus aims at a more precise scaling of the viscosity strength in the presence of discontinuities in the solution. 

%% file: 6_sensor.tex
\section{Shock sensor}
\label{sec:sensor}

In this section, we describe the shock sensor which is utilised in the our subsequent numerical tests.
The shock detection algorithm is based on the one of Persson and Peraire \cite{persson2006sub} and 
does not just flag troubled cells but also steers the viscosity strength $\varepsilon$.

Further, the sensor is based on the rate of decay of the expansion coefficients of the polynomial approximation $u_p$ when represented with respect to a hierarchical orthonormal basis $\{ \varphi_k \}_{k=0}^{K(p)}$ of the approximation space $\mathbb{P}_p\left( \Omega_{ref} \right)$. 
Here, $\Omega_{ref}$ denotes a reference element in which all computations are performed. 
Typical choices are $\Omega_{ref}=[-1,1]$ in one dimension and, for instance, in two dimensions $\Omega_{ref} = [-1,1]^2$ when quadrilateral elements are used or $\Omega_{ref} = \mathbb{T} := \{(x,y) \in \R^2| x\geq 0,y\geq 0, x+y\leq 1\}$ when triangular elements are used. 
Furthermore, $K(p)-1$ denotes the number of basis elements needed to generate the approximation space. 
Thus, when expressed in a hierarchical orthonormal basis $\{ \varphi_k \}_{k=0}^{K(p)}$, the numerical solution reads $u_p = \sum_{k=0}^{K(p)} \hat{u}_k \varphi_k$ on every element. 
The truncated series of reduced degree $p-1$ is represented by $u_{p-1} = \sum_{k=0}^{K(p-1)} 
\hat{u}_k \varphi_k$.
The shock detector is now based on 
\begin{equation}
    S := \frac{\scp{u_p-u_{p-1}}{u_p-u_{p-1}}}{\norm{u_p}^2}
\end{equation}
on every element. Numerical experiments indicate that a hierarchical orthogonal basis is sufficient to define the operator $S$. 
By assuming analogy to Fourier expansions, $S$ measures the degree of smoothness and is expected to scale like 
\begin{equation}
    S \sim p^{-4}
\end{equation}
for continuous $u$ for which $u_p$ is the orthogonal projection into $\mathbb{P}_p\left( \Omega_{ref} \right)$. 
Thus, artificial viscosity should be activated for $S>p^{-4}$, since a discontinuity is 
anticipated then. 
Applying the logarithm on both terms, the above condition can also be written as 
\begin{equation}\label{eq:sensor}
    s := \log(S) > -4 \log(p) =: s_{ref}.  
\end{equation}
While the reference value $s_{ref}$ for continuous $u$ remains constant, the sensor value $s$ linearly increases for increasing smoothness and decreases for decreasing smoothness of $u$. 

Once a shock has been sensed with the help of \eqref{eq:sensor}, and artificial viscosity is activated, the amount of viscosity, i.e. the viscosity strength $\varepsilon$ in \eqref{eq:general_vis}, is determined by the smooth function 
\begin{equation}\label{eq:sensor_function}
    \varepsilon = 
    \begin{cases}
        0 & , s < s_{ref} - \kappa, \\ 
        \frac{\varepsilon_{max}}{2}\left( 1 + \sin \frac{\pi (s - s_{ref})}{2 \kappa} \right) & , 
            s_{ref} - \kappa \leq s \leq s_{ref} + \kappa, \\ 
        \varepsilon_{max} & , s > s_{ref} + \kappa,
    \end{cases}
\end{equation}
with \textit{maximal viscosity strength} $\varepsilon_{max} \sim \frac{h}{p}$ and a problem dependent \textit{ramp parameter} $\kappa > 0$. 
In this work, following \cite{klockner2011viscous}, the ramp parameter is chosen as $\kappa = 1$, while, for scalar problems, the maximal viscosity strength is chosen as 
\begin{equation}
    \varepsilon_{max} 
        = \frac{1}{2} \cdot \max\left( \left| \frac{\partial f}{ \partial u} \right| \right) \cdot 
        \frac{h}{p}.
\end{equation}

Further, following \cite{barter2010shock}, we made a slight modification of the sensor by going over to use 
\begin{equation}\label{eq:modification}
    s = \log F 
    \quad \text{with} \quad 
    F = \min \left( c p^4 \cdot S, 1 \right) 
\end{equation} 
and $s_{ref} = -2.0$ instead of \eqref{eq:sensor}, where $c$ is a suitably chosen parameter 
defined to control the sensor sensitiveness. 
This parameter becomes the only problem dependent 
variable and, in general, the stronger the nonlinearity of the underlying equations is the 
higher the value of $c$ becomes.

Closing this section, we want to note that several other shock sensors have been proposed for the selective application of artificial viscosity as well as mesh refinement. 
These shock detection algorithms, for instance, use information on 
the $L^2$ norm of the residual of the variational form 
\cite{bassi1995accurate,jaffre1995convergence}, 
the primary orientation of the discontinuity \cite{hartmann2006adaptive}, 
magnitude of the facial inter-element jumps \cite{barter2010shock,feistauer2007robust} or 
entropy pairs \cite{guermond2008entropy}. 
Yet, for most of these sensors it is typically unclear which value indicates a shock 
discontinuity or other features of instability. 
A shortcoming of these methods thus is that a variety of scaling choices on a case-by-case basis have to be proposed, where no assignment of the scaled quantity to an explicit meaning is clear anymore.

The above sensor of Persson and Peraire, however, has the advantage of a proper scaling: 
Most of the listed quantities essentially relate to the smoothness of the underlying solution $u$, and thus how well-resolved it is by a polynomial (best) approximation. 
There are also enhancements of the sensor of Persson and Peraire, see for instance the recommended 
work \cite{klockner2011viscous}, which would exceed the scope of this work. 
Since we want to focus on the advantage of using novel viscosities rather than on the underlying shock detection algorithm, the subsequent numerical tests are performed by using the basic but fairly reliable sensor \eqref{eq:sensor_function}. 

%% file: 7_tests.tex
\section{Extension to the Euler equations of gas dynamics}
\label{sec:tests}

In this section, the prior theoretical investigations and results are demonstrated for the system 
of Euler equations. 
The extension of the proposed artificial viscosity method to the system of Euler equations (or any 
other system) is straightforward and described in section \ref{sub:Sod}. 
Numerical tests for the problems of Sods shock tube as well as Shu and Osher's shock tube are 
presented. 
In order to keep the presentation compact, we have decided to demonstrate the subsequent tests just for the super Gaussian viscosity \eqref{eq:super_gaussian} as a prototype of the proposed $C^0$ viscosities.

The numerical tests demonstrate that the proposed $C_0^\infty$ viscosities provide observable 
sharper profiles and steeper gradients than the usual $C^0$ viscosity. 
Shu and Osher's shock tube - especially designed for high-order methods - further illustrates that 
the new $C_0^\infty$ artificial viscosity method, compared to the $C^0$ artificial viscosity method, 
is capable of an advanced representation of wave amplitudes of high frequency features.   

\subsection{Sod's shock tube}
\label{sub:Sod}

We now extend the artificial viscosity method to the Euler equations of gas dynamics 
\begin{equation}\label{eq:euler2}
    \partial_t 
    \underbrace{\begin{pmatrix} \rho \\ m \\ E \end{pmatrix}}_{=u} 
    + \partial_x 
    \underbrace{\begin{pmatrix} m \\ vm + \rho P \\ v (E + P) \end{pmatrix}}_{=f(u)}
    = 0,
\end{equation}
where $\rho$ is the density, $v$ is the velocity, $E$ is the total energy, $m = \rho v$ is the momentum, $P = (\gamma - 1)(E-\frac{1}{2}v^2 \rho)$ is the pressure, and $\gamma = 1.4$ is taken for the ratio of specific heat. 
The Euler equations express the conservation of mass, momentum, and energy for a perfect gas. 
The extension of the artificial viscosity is straightforward, since the method and the prior sensor is applied to the density variable $u_1 = \rho$ and, once the shock is detected, the viscosity is added to all conserved variables. 

Figure \ref{fig:Sod} reports the results of this procedure for the standard test case of Sod's shock tube 
\begin{equation}\label{eq:Sod}
    \left\{
    \begin{aligned}
        & \rho = \rho_L = 1,     && P = P_L = 1,   &&& v = v_l = 1, &&&& \text{ if } x < 0.5, \\ 
        & \rho = \rho_R = 0.125, && P = P_R = 0.1, &&& v = v_R = 0, &&&& \text{ if } x > 0.5, 
    \end{aligned}
    \right.
\end{equation}
on $\Omega=[0,1]$ at time $t = 0.2$.

\begin{figure}[!htb]
  \centering
  \begin{subfigure}[b]{0.5\textwidth}
    \includegraphics[width=\textwidth]{%
      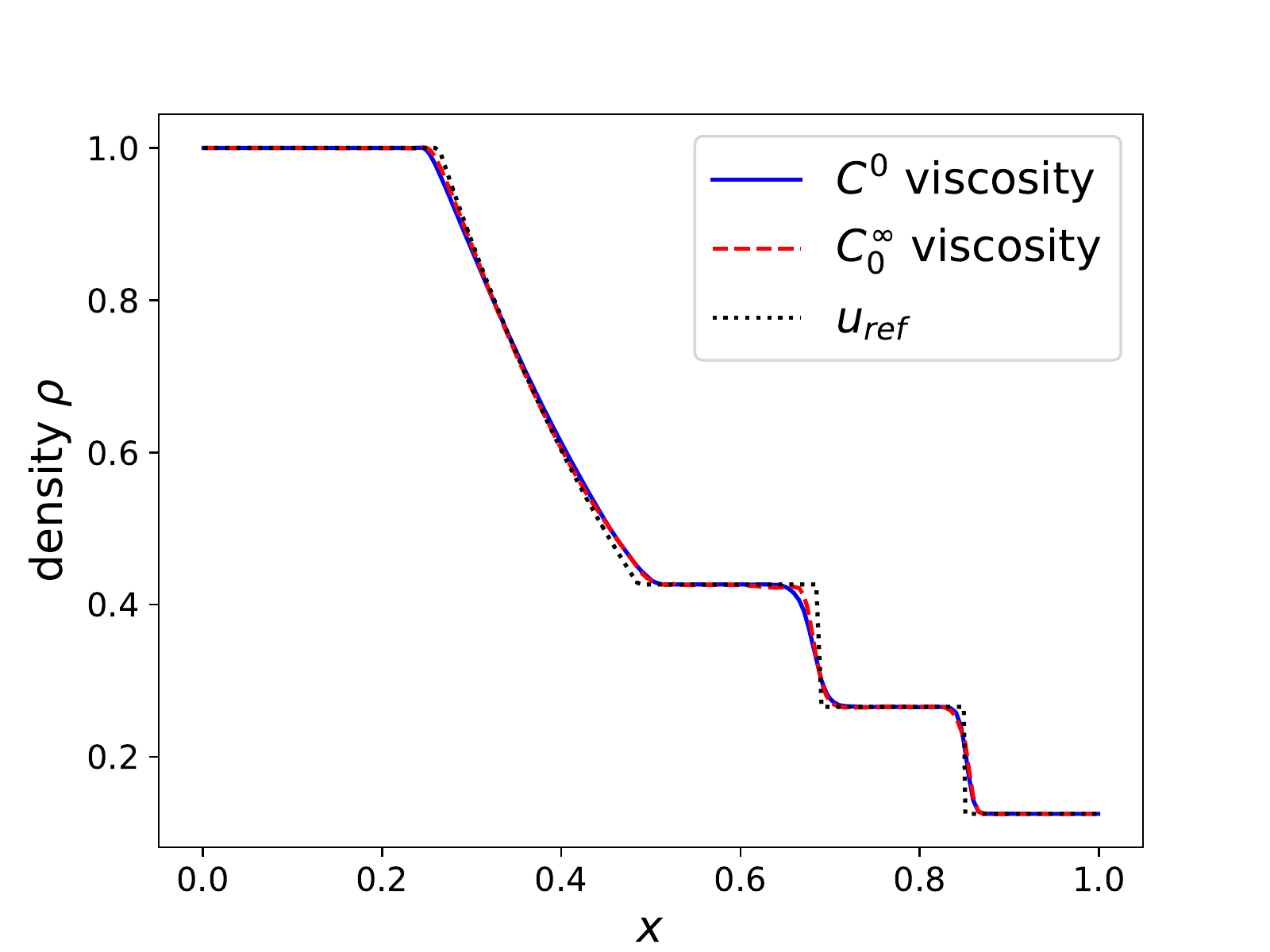}
    \caption{Density}
    \label{fig:Sod_density}
  \end{subfigure}%
  ~
  \begin{subfigure}[b]{0.49\textwidth}
    \includegraphics[width=\textwidth]{%
      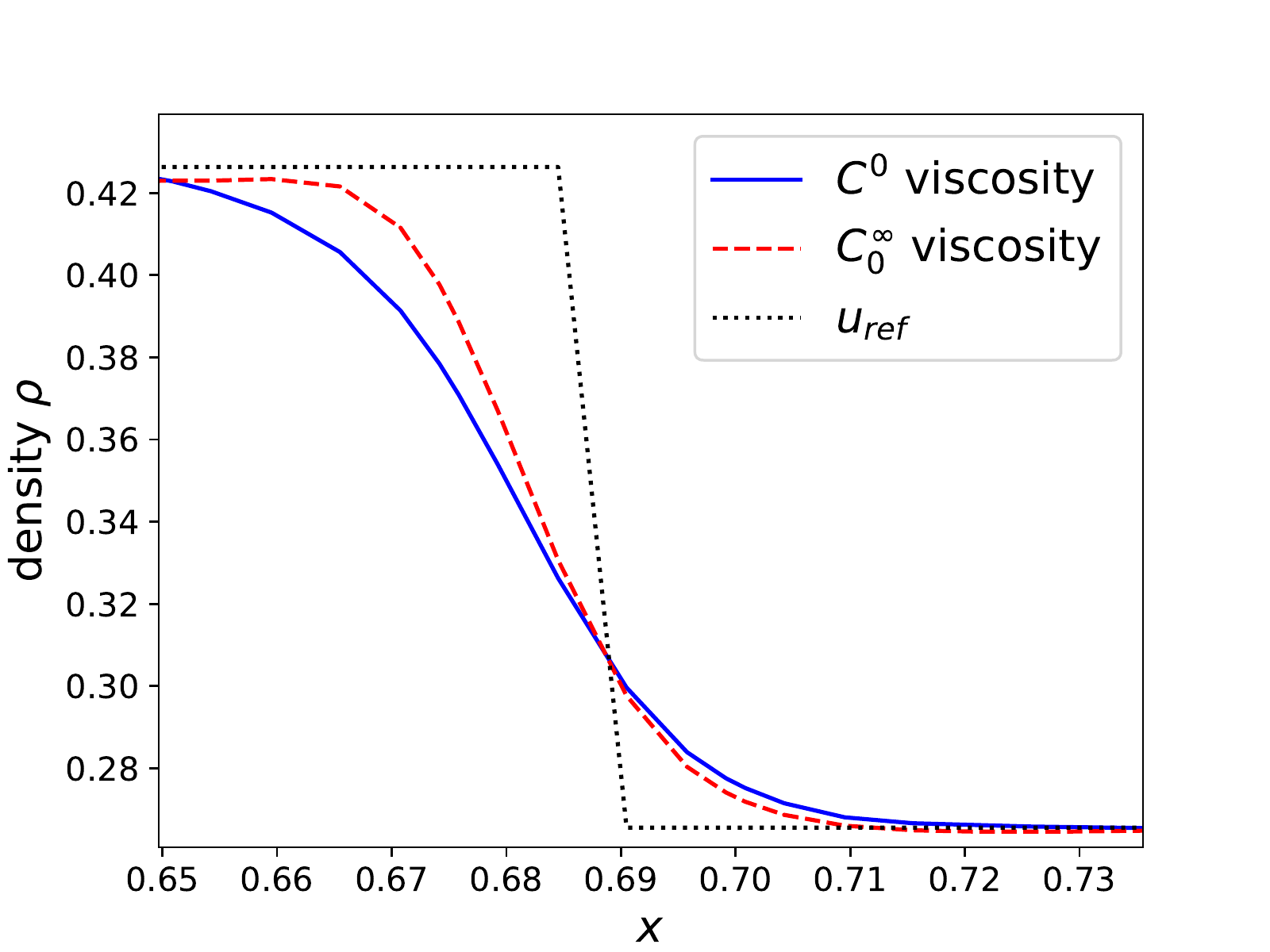}
    \caption{Density - zoom}
    \label{fig:Sod_density_zoom}
  \end{subfigure}%
  \\
  \begin{subfigure}[b]{0.49\textwidth}
    \includegraphics[width=\textwidth]{%
      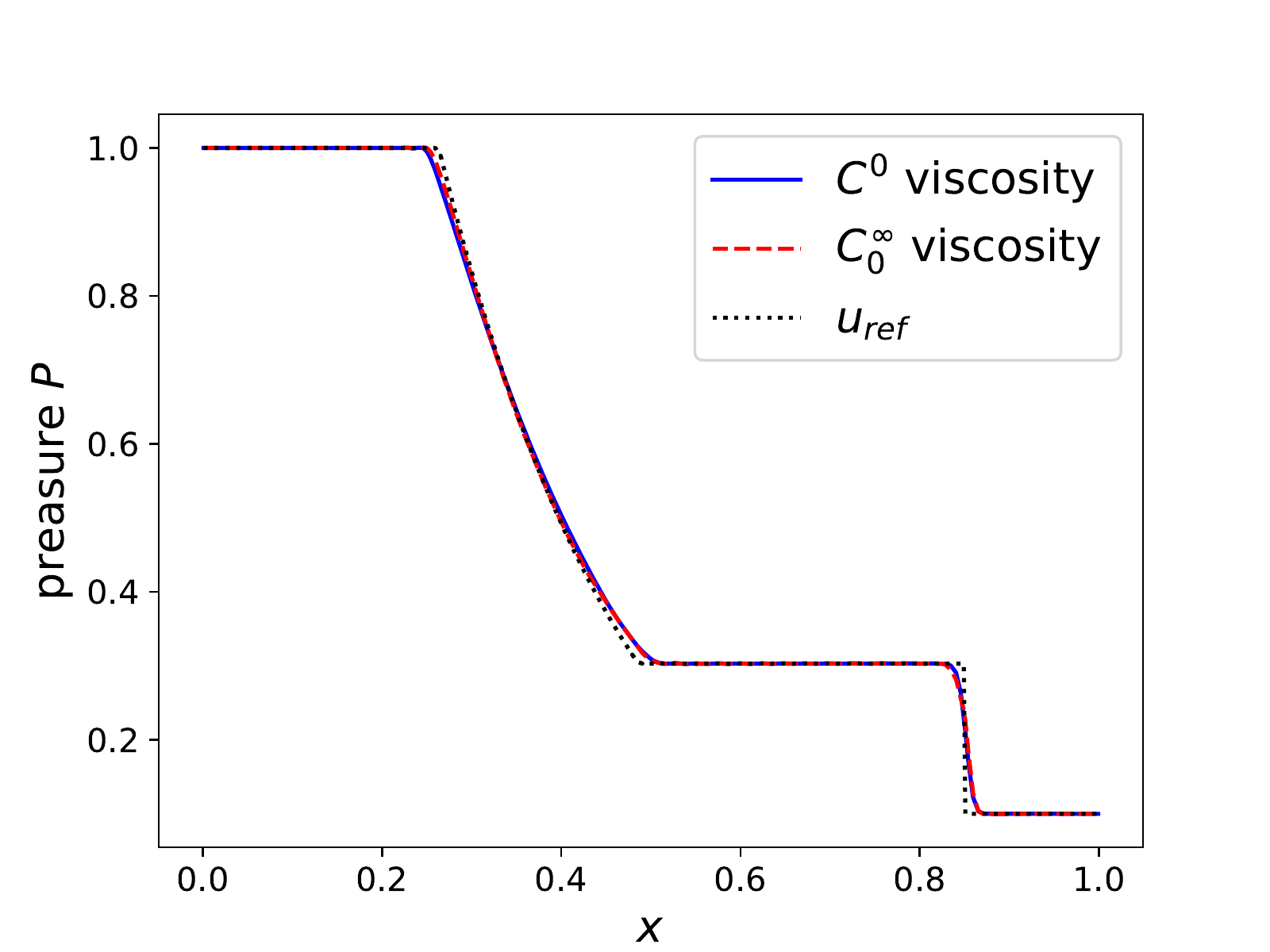}
    \caption{Pressure}
    \label{fig:Sod_preasure}
  \end{subfigure}%
  ~
  \begin{subfigure}[b]{0.49\textwidth}
    \includegraphics[width=\textwidth]{%
      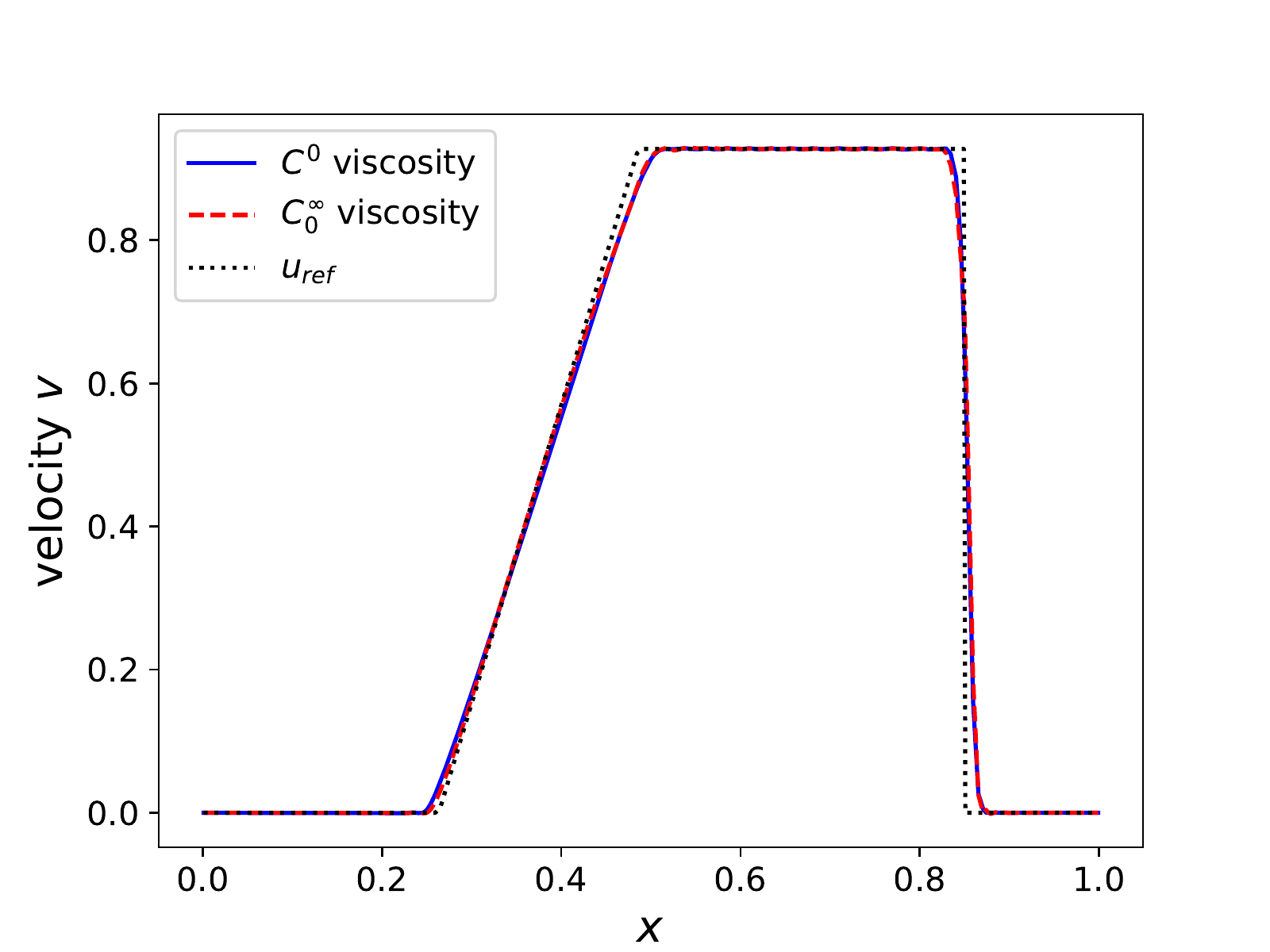}
    \caption{Velocity}
    \label{fig:Sod_velocity}
  \end{subfigure}%
  \caption{Numerical solutions for Sod's shock tube.}
  \label{fig:Sod}
\end{figure}

The numerical solutions were obtained using $I=40$ equidistant elements and polynomial degree $p=5$.
We decided for Sod's shock tube as a first test case since this is a relatively mild test case for 
which an analytical reference solution can still be determined, \cite{sod1978survey}.
The profiles for density, pressure, and velocity are illustrated in Figure \ref{fig:Sod} by (black) 
dotted lines and consist of a left rarefaction, a contact, and a right shock.

At the same time, the profiles of numerical solutions are realised by (blue) straight lines for 
the $C^0$ artificial viscosity method and by (red) dashed lines for the super Gaussian $C_0^\infty$ 
artificial viscosity method. 
For the super Gaussian viscosity \eqref{eq:super_gaussian} the parameters $\lambda = 100$ and 
$\alpha = -\ln\left(10^{-16}\right)$ were chosen.
We made no effort to optimise these parameters, which appeared to be fairly robust in our tests. 
More suitable choices of parameters, yielding to further enhanced results, thus seem possible. 
The numerical solution obtained by the super Gaussian $C_0^\infty$ artificial viscosity method shows 
a sharper profile and steeper gradients; particularly visible in the density profile at the 
contact discontinuity around $x\approx 0.68$, see Figure \ref{fig:Sod_density_zoom}.

\subsection{Shu and Osher's shock tube}

The test case of Sod's shock tube \eqref{eq:Sod} is reasonable when it is intended to 
demonstrate how a method handles different types of discontinuities. 
Yet, no small-scale features were present. 
A more challenging test case to observe if a method is able to capture both, shocks as well as 
small-scale smooth flows, is Shu and Osher's shock tube 
\begin{equation}\label{eq:Shu_OSher} \small
    \left\{
    \begin{aligned}
        & \rho = 3.857143, && P = 10.33333, &&& v = 2.629369, 
        &&&& \text{if } x < 0.5, \\
        & \rho = 1 + 0.2 \sin\left(5 \pi x\right), && P = 1, &&& v = 0, 
        &&&& \text{if } x > 0.5, 
    \end{aligned}
    \right.
\end{equation}
for the Euler equations \eqref{eq:euler2} on $\Omega = [-5,5]$, 
which was proposed in \cite{shu1989efficient}.  
Figure \ref{fig:Shu_Osher} shows the profiles of the numerical solutions for density, 
pressure, and velocity at time $t=1.8$. 

\begin{figure}[!htb]
  \centering
  \begin{subfigure}[b]{0.5\textwidth}
    \includegraphics[width=\textwidth]{%
      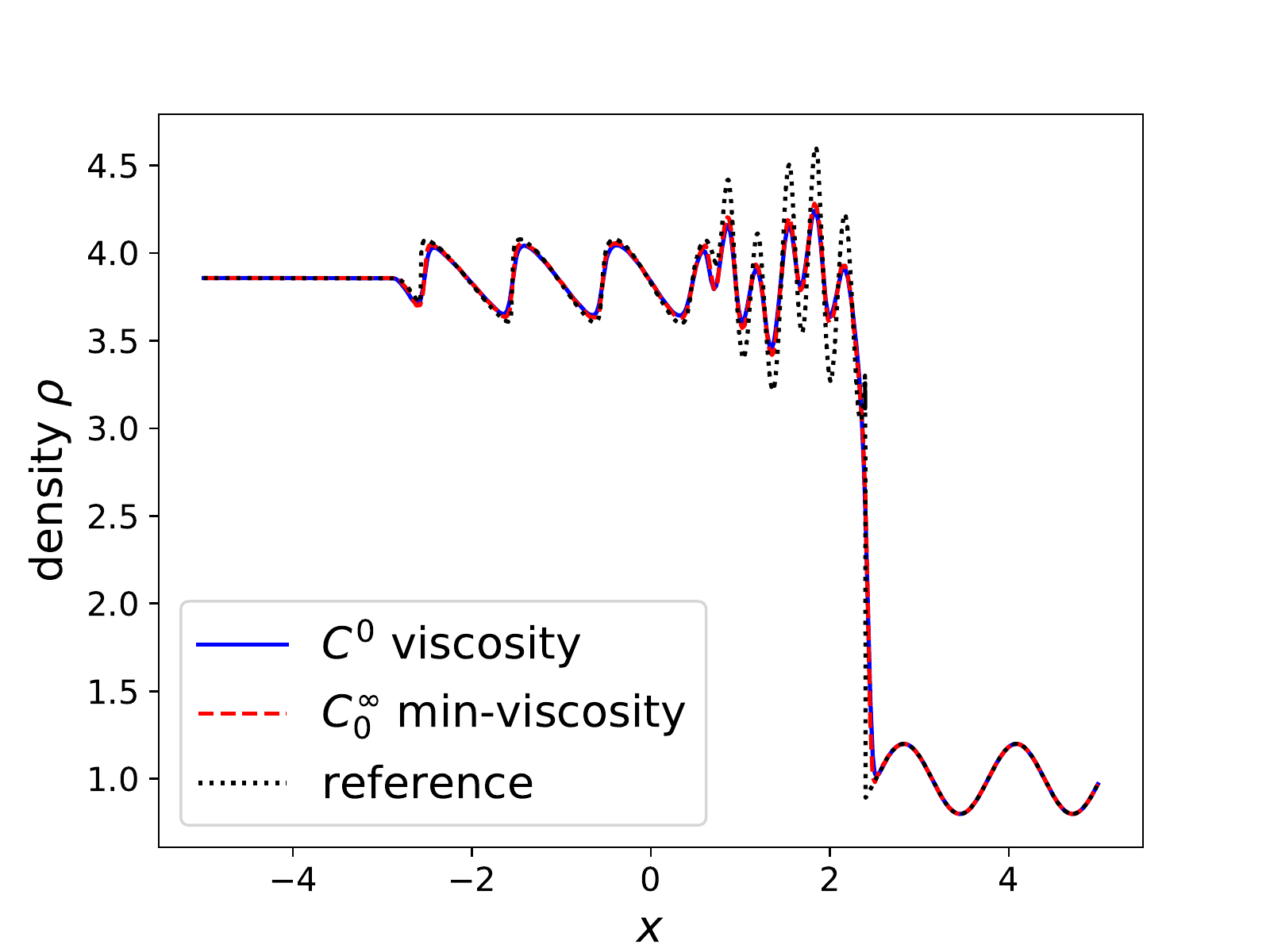}
    \caption{Density}
    \label{fig:Shu_Osher_density}
  \end{subfigure}%
  ~
  \begin{subfigure}[b]{0.49\textwidth}
    \includegraphics[width=\textwidth]{%
      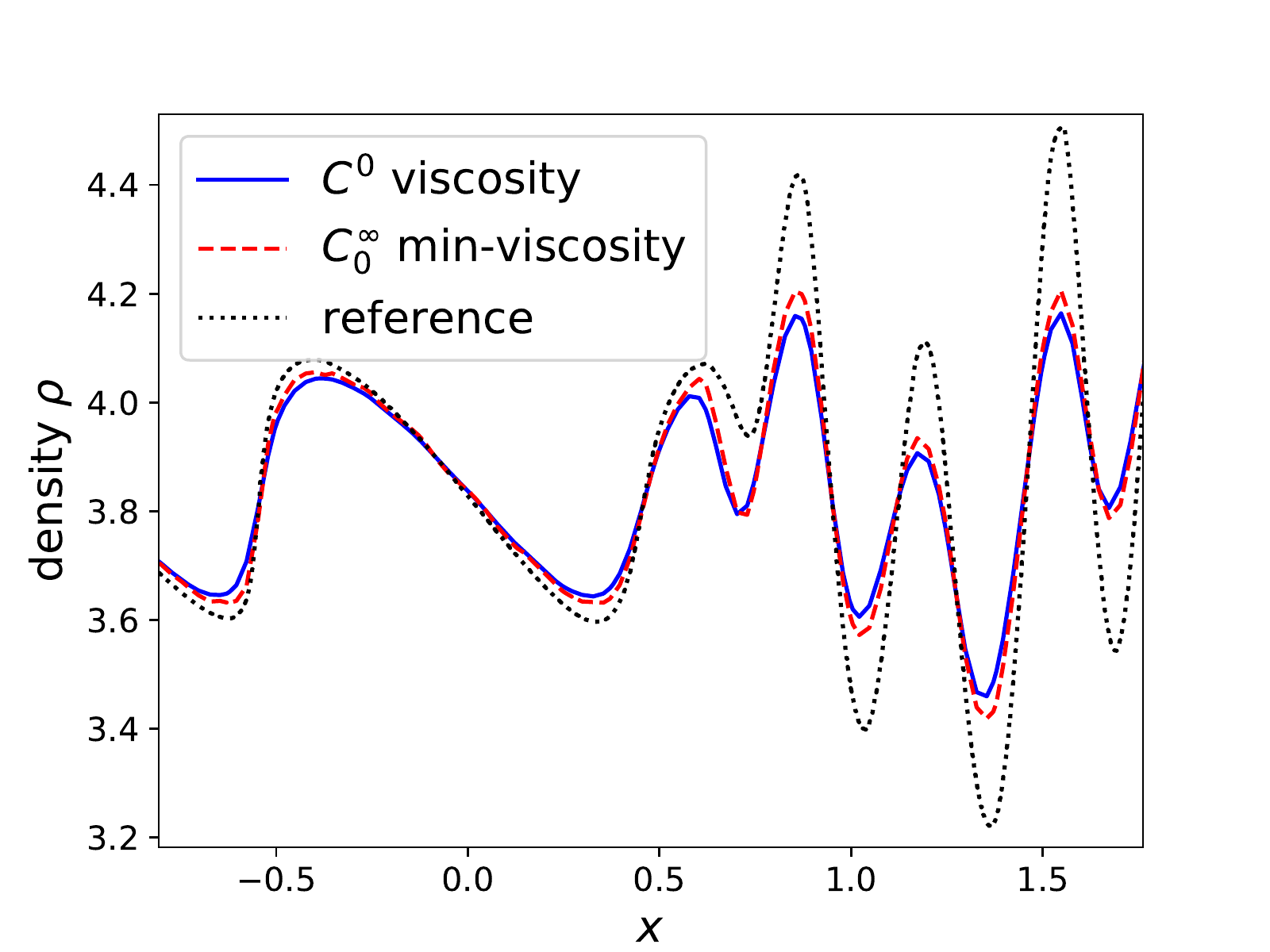}
    \caption{Density - zoom}
    \label{fig:Shu_Osher_density_zoom}
  \end{subfigure}%
  \\
  \begin{subfigure}[b]{0.49\textwidth}
    \includegraphics[width=\textwidth]{%
      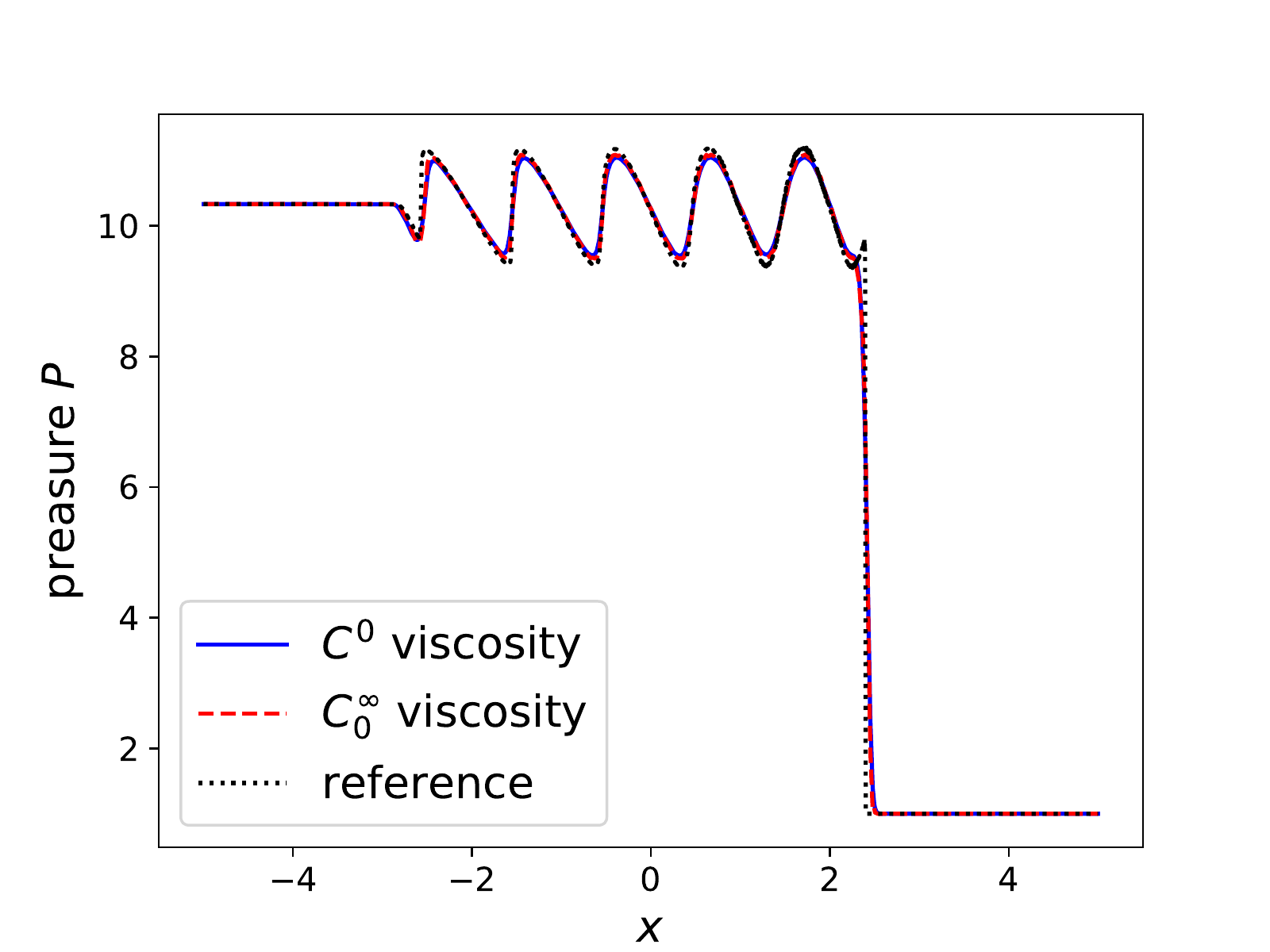}
    \caption{Pressure}
    \label{fig:Shu_Osher_preasure}
  \end{subfigure}%
  ~
  \begin{subfigure}[b]{0.49\textwidth}
    \includegraphics[width=\textwidth]{%
      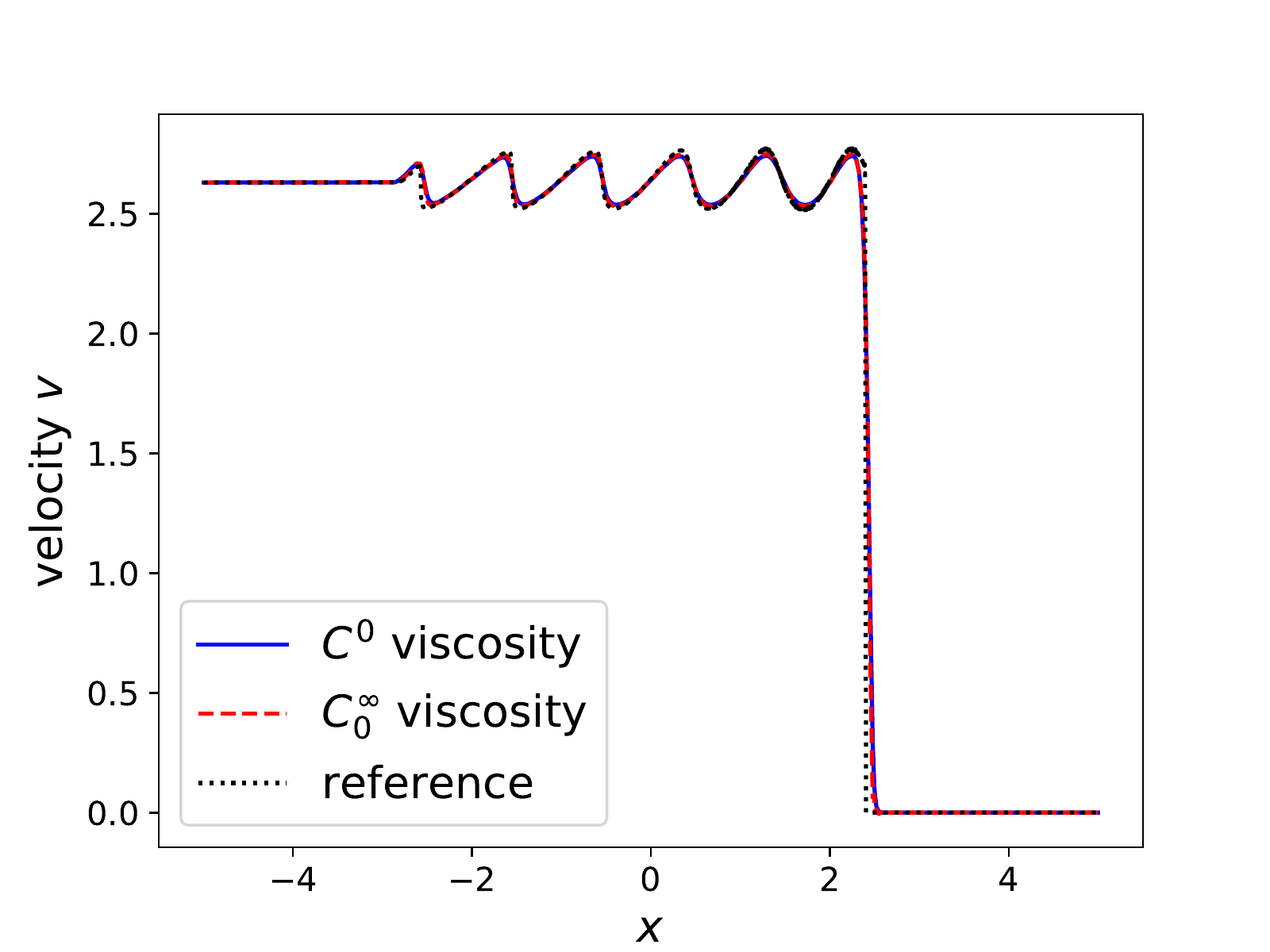}
    \caption{Velocity}
    \label{fig:Shu_Osher_velocity}
  \end{subfigure}%
  \caption{Numerical solutions for Shu and Osher's shock tube.}
  \label{fig:Shu_Osher}
\end{figure}

The numerical solutions by $C^0$ artificial viscosity - (blue) straight line - 
and by super Gaussian $C_0^\infty$ artificial viscosity - (red) dashed line - are both computed 
using $I=80$ equidistant elements and polynomial approximations of degree $p=5$.
The reference solution - (black) dotted line - is computed with polynomial degree $p=1$, $I 
=10 \ 000$ elements, and the generalized slope limiter, 
see  
\cite{cockburn1991runge,cockburn1989tvb,cockburn1989tvb2,cockburn1990runge,cockburn1998runge,cockburn1999discontinuous} 
or \cite[page 152]{hesthaven2007nodal}.
We decided for the test case of Shu and Osher's shock tube, 
since it demonstrates the advantage of a reasonable chosen amount of dissipation, when the problem 
involved has some structure. 
In fact, we can observe an enhanced representation of the small-scale wave amplitudes, 
especially left from the shock discontinuity in the profile of the density, for the novel super 
Gaussian $C_0^\infty$ artificial viscosity method compared to the usual $C^0$ artificial viscosity 
method. 
This is illustrated in greater detail in Figure \ref{fig:Shu_Osher_density_zoom}.
For instance, in long time or large eddy simulations, the preservation of such small-scale features 
is highly desired.   

\begin{remark}
    The above results raise the question of the effect of the proposed viscosity on the spectral 
resolution of the method, in particular compared to usual viscosities. 
    Future work will investigate these spectral properties by means of a (nonlinear) spectral 
analysis as described, for instance, in 
    \cite{pirozzoli2006spectral,fauconnier2011spectral,wheatley2011spectral}. 
    We think it would be of interest to have a more rigorous study on how the spectral resolution 
of the artificial viscosity method depends on the viscosity function.
\end{remark}

%% file: Summary.tex
\section{Summary and conclusions}
\label{sec:summary}

In this work, a novel artificial viscosity method has been introduced utilising smooth and compactly supported viscosity distributions. 
In order to derive them, widely used artificial viscosity methods, such as the ones of Persson and Peraire as well as Kl\"ockner et al., were analytically revisited with respect to the essential design criteria of conservation and entropy stability. 
It was proved for the viscosity extension that conservation carries over if the viscosity is continuous and compactly supported, while entropy stability already holds for positive viscosities. 

Further investigating the method of modal filtering, it was demonstrated that this strategy has inherent shortcomings, which are related to the nature of Legendre and more general Jacobi viscosities. 
Since these have their peak in the middle of elements and rapidly decrease away from the element center, problems arise for shock discontinuities near element boundaries. 

Overcoming this drawback, the new $C_0^\infty$ were constructed such that they are 
approximately constant over nearly the whole element. 
Smooth and compactly supported functions with $\approx 1$ over most of an element can be found in 
the field of robust reprojection as well as mollifiers and were proposed as viscosity distributions 
for the first time in this work. 

Numerical tests for the Euler equations demonstrated the novel (super Gaussian) $C_0^\infty$ artificial viscosity to provide sharper profiles, steeper gradients and a higher resolution of small-scale features while still maintaining stability of the method. 

Since all artificial viscosity methods heavily rely on trustworthy detection of discontinuities, further research on shock sensor strategies is mandatory. 
It is our opinion that especially shock sensors which are capable of detecting the precise location 
and strength of jump discontinuities, such as the concentration method of Gelb and Tadmor 
\cite{gelb1999detection,gelb2000detection,gelb2006adaptive,gelb2008detection,offner2013detecting} or 
polynomial 
annihilation \cite{archibald2005polynomial,archibald2008determining}, seem highly promising.
This would also allow to moreover adapt the artificial viscosity method in every element to the exact location of a shock. 
Dissipation would just be added where it is needed.